\documentclass[pdflatex,sn-mathphys-num]{sn-jnl}% Math and Physical Sciences Numbered Reference Style
\usepackage{graphicx}
\usepackage{lipsum}
\usepackage{amsfonts}
\usepackage{epstopdf}
\usepackage[shortlabels]{enumitem}
\usepackage{microtype}
\usepackage[caption=false]{subfig}
\usepackage{booktabs}
\usepackage{amsthm}
\usepackage{hyperref}
\usepackage{float}
\usepackage{url}
\usepackage{amsmath}
\usepackage{amssymb}
\usepackage{mathtools}
\usepackage{physics}
\usepackage{xcolor}
\usepackage{bbm}
\usepackage{soul}
\usepackage{multirow}
\usepackage{adjustbox}
\usepackage[algo2e,ruled,linesnumbered]{algorithm2e}
\usepackage{algorithm}
\usepackage{algorithmic}
\usepackage[capitalize]{cleveref}
\usepackage{cases}
\usepackage{natbib}
\usepackage{tabularray}
\usepackage{diagbox}
\usepackage{longtable}
\usepackage{rotating}
\usepackage{mathrsfs}
\usepackage[title]{appendix}
\usepackage{textcomp}
\usepackage{manyfoot}
\usepackage{listings}
\usepackage{longtable}

\usepackage{tikz}
\usetikzlibrary{tikzmark,calc}

\theoremstyle{thmstyleone}%
\newtheorem{theorem}{Theorem}%  meant for continuous numbers
\newtheorem{proposition}[theorem]{Proposition}% 
\newtheorem*{proposition*}{Proposition}% 
\AddToHook{env/proposition/begin}{\crefalias{theorem}{proposition}}

\newtheorem{lemma}[theorem]{Lemma}
\AddToHook{env/lemma/begin}{\crefalias{theorem}{lemma}}

\theoremstyle{thmstyletwo}%
\newtheorem{remark}{Remark}%

\theoremstyle{thmstylethree}%
\newtheorem{definition}{Definition}%

\usepackage{amsmath,amsfonts,bm}

\def\1{\bm{1}}

\DeclareMathAlphabet{\mathsfit}{\encodingdefault}{\sfdefault}{m}{sl}
\SetMathAlphabet{\mathsfit}{bold}{\encodingdefault}{\sfdefault}{bx}{n}

\def\gC{{\mathcal{C}}}
\def\gD{{\mathcal{D}}}

\def\gM{{\mathcal{M}}}
\def\gN{{\mathcal{N}}}

\def\gT{{\mathcal{T}}}

\def\sR{{\mathbb{R}}}

\newcommand{\R}{\mathbb{R}}

\newcommand{\KLz}{Kurdyka--Łojasiewicz{~}\allowbreak}
\newcommand{\KL}{\mathrm{D}_{\mathrm{KL}}}

\newcommand{\Rn}{\sR^n}
\newcommand{\Rm}{\sR^m}

\newcommand{\beq}{\begin{equation}}
\newcommand{\eeq}{\end{equation}}

\newcommand{\Id}{\mathrm{Id}}

\newcommand{\usol}{u^{\star}}

\newcommand{\wsol}{w^{\star}}

\newcommand{\xk}{x^{(k)}}
\newcommand{\xkp}{x^{(k+1)}}
\newcommand{\xsol}{x^{\star}}
\newcommand{\xhat}{\hat{x}}

\newcommand{\ykp}{y^{(k+1)}}

\newcommand{\zk}{z^{(k)}}
\newcommand{\zkp}{z^{(k+1)}}

\newcommand{\laml}{\lambda_l}
\newcommand{\lamg}{\lambda_g}
\newcommand{\lamp}{\lambda_p}

\newcommand{\sgo}{\sigma_1}
\newcommand{\sgt}{\sigma_2}
\newcommand{\pso}{\phi_{1}}
\newcommand{\pst}{\phi_{2}}

\newcommand{\dysgs}{PnP-DYS\textsuperscript{GS-GS}}
\newcommand{\dystv}{PnP-DYS\textsuperscript{GS-TV}}

\DeclareMathOperator*{\argmax}{arg\,max}
\DeclareMathOperator*{\argmin}{arg\,min}

\DeclareMathOperator{\prox}{prox}

\DeclareMathOperator{\dist}{dist}

\DeclareMathOperator{\TV}{TV}

\DeclareMathOperator{\dom}{dom}
\DeclareMathOperator{\total}{total}
\DeclareMathOperator{\zer}{zer}
\newcommand{\infconv}{\mathbin{\square}}
\newcommand{\iprod}[2]{\langle #1,\,#2 \rangle}

\newcommand{\sfrac}[2]{%
    {\mathsmaller{\frac{\raisebox{0.05em}{\footnotesize $#1$}}{\raisebox{-0.15em}{\footnotesize $#2$}}}}
}

\begin{document}

\title[Mixed-Noise Plug-and-Play with Infimal Convolution Fidelities and Multiple Priors]{Mixed-Noise Plug-and-Play with Infimal Convolution Fidelities and Multiple Priors}


\author[1]{\fnm{Ziqi} \sur{Qin}}\email{zq245@cam.ac.uk}

\author*[2]{\fnm{Hong Ye} \sur{Tan}}\email{hyt35@math.ucla.edu}

\author[1]{\fnm{Ander} \sur{Biguri}}\email{ab2860@cam.ac.uk}

\author[1]{\fnm{Jingwei} \sur{Liang}}\email{jl993@cam.ac.uk}

\author[1]{\fnm{Carola--Bibiane} \sur{Sch\"onlieb}}\email{cbs31@cam.ac.uk}

\affil[1]{\orgdiv{Department of Applied Mathematics and Theoretical Physics}, \orgname{University of Cambridge}, \orgaddress{\street{Wilberforce Road}, \city{Cambridge}, \postcode{CB3 0WA}, \state{Cambridgeshire}, \country{United Kingdom}}}

\affil*[2]{\orgdiv{Department of Mathematics}, \orgname{University of California Los Angeles}, \orgaddress{\street{520 Portola Plaza}, \city{Los Angeles}, \postcode{90024}, \state{California}, \country{USA}}}

% \affil[3]{\orgdiv{Department}, \orgname{Organization}, \orgaddress{\street{Street}, \city{City}, \postcode{610101}, \state{State}, \country{Country}}}


\abstract{Plug-and-Play (PnP) algorithms are a class of iterative methods for inverse imaging. Within an optimization algorithm, they combine a flexible fidelity term, encoding the forward operator, and a pretrained image denoiser, in order to deal with more severe corruptions such as blurring or downsampling when reconstructing an image. This work studies provably convergent PnP methods for mixed-noise forward processes by using the infimal convolution as a fidelity term, providing a statistical interpretation as a joint maximum a-posteriori estimator over both noises, and preserving the Bayesian MAP interpretation of PnP methods. Independently, we extend the PnP formulation to multiple prior terms using the Davis--Yin three-operator splitting. This extension can be combined with either standard fidelities or the proposed mixed-noise infimal-convolution fidelities. We verify that these generalized PnP methods are convergent under standard \KLz conditions. Numerical experiments on Laplace-Gaussian and Poisson--Gaussian noise demonstrate stable convergence of single-prior and multiple-prior PnP methods, and divergence under fidelity mismatch. Furthermore, PnP with infimal convolution fidelities are able to scale to noise up to 38\% standard deviation, with multiple priors reaching different stationary points that qualitatively preserve more textural properties.}




\maketitle

\section{Introduction}

Many imaging applications can be formulated as linear inverse problems, with prominent examples including image restoration, computed tomography in medical imaging, and astrophysics. For a linear forward operator $A:\Rn\rightarrow \Rm$ and a noise degradation operator $\gT:\Rm\rightarrow \Rm$, the problem is to find $x\in\Rn$ from a noisy measurement $y\in\Rm$ given by the forward model
\begin{equation*}\label{eq:forward}
y=\gT(Ax).
\end{equation*}
A common approach is \textit{variational regularization} \cite{benning2018modern}, where the solution is given as the minimizer of a (convex) optimization problem 
\begin{align*}
    \hat x = \argmin_{x \in \Rn} f(x;y) + g(x),
\end{align*}
where $f(x;y)$ is a data fidelity function that measures how well $\hat x$ satisfies the measurement, and $g(x)$ is a regularization term encoding prior knowledge of the solution. The most common noise model for $\gT$ is \textit{additive white Gaussian noise}, where the noisy measurement can be modeled as $y=\gT(Ax) = Ax+\omega$, where $\omega\sim \gN(0, \sigma^2 \Id)$ for some (known) noise level $\sigma$. This noise model corresponds to the squared fidelity term $f(x) = \|Ax - y\|^2/(2\sigma^2)$, and denoisers targeting Gaussian noise precisely target the conditional mean of the clean image given the noisy observation \cite{tan2025solving}. A wide variety of such denoisers exists in the literature, such as the classical nonlocal means or BM3D denoisers \cite{buades2005non,dabov2007image}, or recent denoisers based on deep neural networks \cite{zhang2017learning,zhang2017beyond}.

Gaussian noise models are often insufficient to model physical phenomena. A common alternative noise model is Poisson noise, where $\gT(Ax) \sim \operatorname{Pois}(Ax)$, occurring in physical low-photon imaging problems such as fluorescence microscopy or emission tomographic imaging \cite{morris2015imaging,timmermann2002multiscale,melidonis2023efficient}. In digital image transmission or satellite imaging, images may suffer from impulsive noise \cite{chan2005salt} or heavy-tailed distributions \cite{achim2003sar}. From a reconstruction point of view, the fidelity terms corresponding to the likelihood of these noise models often have poor regularity properties, such as exploding Lipschitz constants near hard constraints, arising from physical constraints such as pixel non-negativity. 

In addition to noise being non-Gaussian, images are rarely affected by only a single noise source. Instead, models often consider corruption by a combination of noise sources arising from different physical processes. For instance, in fluorescence microscopy and low-dose CT, the signal is typically corrupted by a mixture of Poisson shot noise inherent in photon counting, and Gaussian read-out noise originating from sensor electronics \cite{mannam2022real,chen2017low,sarder2006deconvolution}.

Solving mixed noise problems requires combining various fidelity terms together, manifesting as multiple terms in a variational formulation \cite{xiao2011restoration,benvenuto2008study}. Bayesian approaches are a particularly attractive framework to model mixed noise processes, and usually target either the minimum mean-squared error (MMSE) estimator \cite{luisier2010image,le2014unbiased}, or the maximum a-posteriori (MAP) estimator which is equivalent to a variational problem. Using Bayes' rule for the MAP estimator yields the \textit{infimal convolution} framework of \cite{calatroni2017infimal}, deriving classical regularization guarantees in the case of Laplace--Gaussian or Poisson--Gaussian noise. This provides a principled fidelity term that encodes the maximum joint likelihood of both noise processes.

Given some fidelity term that encodes the measurement consistency for mixed noise, it remains to define a regularization term. \textit{Plug-and-Play} (PnP) methods provide an implicit way of doing so, defined by an image denoiser. This work proposes to use the PnP framework in order to target mixed noise problems, first by utilizing the infimal convolution fidelity \cite{calatroni2017infimal}, and further by using the three-operator splitting of Davis and Yin \cite{davis2017three} to utilize multiple denoisers. This combines the exact MAP interpretation of the infimal convolution fidelity with the tuneable and expressive reconstruction power of PnP methods.

\section{Related work and contributions}

\subsection{Plug-and-Play (PnP) Methods}
The first Plug-and-Play method was introduced by Venkatakrishnan et al. for image reconstruction, motivated by splitting algorithms for variational regularization problems \cite{venkatakrishnan2013plug}. The main observation is that solving $\min_x f+g$ can be done by alternating descent steps on $f$ and $g$ individually, using algorithms such as the alternating direction method of multipliers (ADMM). The descent step on the prior/regularization term can be interpreted as ``improving the image'', or simply image denoising. In this way, the task of hand-crafting a prior can be replaced with simply choosing a suitable image denoiser, demonstrated in \cite{venkatakrishnan2013plug} to be reasonable for simple tomographic problems. For an overview of applications of PnP, we refer to \cite{kamilov2023plug}.

While the observation above seemingly trades the optimization perspective for flexibility and empirical performance, the PnP framework is intricately linked to composite convex optimization. In particular, a PnP method can be defined by inserting a denoiser into a first-order splitting algorithm, such as ADMM \cite{chan2016plug}, proximal gradient descent \cite{ryu2019plug}, Douglas--Rachford splittings (DRS) \cite{hurault2022proximal}, and half-quadratic splittings \cite{zhang2021plug}. The choice of denoiser is flexible; early methods used the aforementioned non-local means and BM3D algorithms, while recently deep learning models such as DnCNN \cite{zhang2017beyond} and DRUNet \cite{zhang2021plug} have become the standard due to their superior performance. Related works incorporating general denoisers with some data fidelity include the regularization-by-denoising framework \cite{romano2017little}, using flow-matching models or diffusion models in place of the denoiser \cite{martin2025pnp, kawar2022denoising,wu2024principled}, or using denoisers within stochastic splittings to sample from a distribution \cite{laumont2022bayesian}. Proximal denoising steps can also be approximated with inner loops, such as unrolled dual forward-backward networks \cite{repetti2022dual}, or dictionary based synthesis denoisers \cite{kowalski2026analysis}.

Theoretical analysis of PnP methods can be adapted from proofs for splitting methods, as a desirable property is whether or not the algorithm converges to a fixed point. This allows for a natural choice of stopping time, as the performance of earlier works depends heavily on the number of iterations. The initial work of \cite{chan2016plug} shows convergence of PnP-ADMM for a family of ``bounded denoisers'' with decreasing step size. Towards more standard assumptions, \cite{ryu2019plug} establishes convergence of various PnP algorithms when the fidelity $f$ is strongly convex and the denoiser $D_\sigma$ satisfies $D_\sigma - I$ being Lipschitz with sufficiently small Lipschitz constant. In general, the denoiser has to satisfy some Lipschitz assumption, such as an ``averaged denoiser'' assumption where $(1-\theta) \Id + \theta D_\sigma$ is nonexpansive for some $\theta\in (0,1)$, or a demicontractive assumption \cite{cohen2021regularization}.

A theoretically attractive denoiser parameterization focuses on gradient step (GS) denoisers, where the denoiser $\gD_{\sigma}$ is defined as the gradient of a scalar-valued potential \cite{cohen2021has,hurault2021gradient}, typically written as $\gD_\sigma = \Id - \nabla g_\sigma$. 

A crucial observation in \cite{hurault2022proximal} is that if $g_\sigma$ has $L$-Lipschitz gradient where $L<1$, then the denoiser takes the proximal form $\gD_{\sigma} = \prox_{\phi_{\sigma}}$ for a potentially non-convex function $\phi_{\sigma}$. Using this, the previous fixed-point theory can be strengthened by characterizing the fixed points, namely, the PnP methods converge to a critical point of the non-convex function $f + \phi_\sigma$. Recent extensions utilizing this identity include quasi-Newton accelerated PnP \cite{tan2023provably} and relaxed GS-PnP schemes for denoisers with Lipschitz constants of $\nabla g_\sigma$ greater than $1$ \cite{hurault2021gradient}. For a comprehensive survey of the latest theoretical advancements in PnP, we refer the reader to \cite{tan2025solving}.

\subsection{Mixed noise and regularization}
Variational approaches can be directly related to maximum a-posteriori (MAP) estimation in the Bayesian perspective. For a likelihood $\ell(y | x) \propto \exp(-f(x;y))$ arising from the noise degradation, and a given prior $\pi(x) \propto \exp(-g(x))$, Bayes' theorem states that the posterior $p(x|y) \propto \ell(y|x) \pi(x)$. The MAP estimator, defined to be $\argmax_x p(x|y)$, is then equivalent to the variational reconstruction $\argmin_x f(x;y) + g(x)$ by taking logs. This can be used to derive variational methods to solve noise with heavy-tailed distributions like the Cauchy noise \cite{sciacchitano2015variational}, multiplicative noise \cite{aubert2008variational,rudin2003multiplicative}, or low-photon imaging \cite{chan2016plug}. Different noises essentially require different fidelity functions.

Towards PnP approaches for non-Gaussian noise, initial works like \cite{rond2016poisson} consider transforming the Poisson-distributed noise into approximately Gaussian noise using the Anscombe transform. To be provably convergent, PnP methods for non-Gaussian noises typically require modifications, such as to avoid denoisers pushing iterates into infeasible regions like negative pixels for Poisson noise. For example, \cite{hurault2023bregman,klatzer2025efficient} utilizes Bregman proximal steps and boundary reflection for reconstruction and sampling. Inexact fidelity functions can also be used to yield provable convergence, but losing the exact likelihood structure given from the true fidelity \cite{modrzyk2026convergent}.

For PnP methods targeting mixed noise, existing strategies mainly consider modified fidelity functions, approximating the mixed noise distribution with another more tractable distribution. The authors in \cite{biquard2025pg} consider high-count Poisson--Gaussian noise by approximating it as Gaussian with spatially varying variance, while \cite{xu2025mixed} consider Gaussian-impulse noise with generalized Gaussian distributions.

Instead of approximating mixed noise with another distribution, we propose PnP through a joint MAP estimation over both the image and latent noise variables, based on the general framework of \textit{infimal convolutions} \cite{calatroni2017infimal,chambolle1997image}. Alternative approaches avoiding distributional approximations include the exact log-likelihood for mixed Poisson--Gaussian noise \cite{chouzenoux2015convex} (up to approximation of an infinite sum). For example, when the image is corrupted by two independent additive noises with one being Gaussian, the resulting data-fidelity term is equivalent to a Moreau envelope, providing a smooth, differentiable expression of the original fidelity function. 

In addition to modifying the fidelity, another possible method to improve reconstruction quality is to use multiple regularizers. For example, the total generalized variation of \cite{bredies2010total} uses sums of higher order (discrete) derivatives within image regularization, reducing staircasing effects from pure total variation regularization. Another natural scenario where this arises is when the image can be decomposed as a sum of two independently regularized components, such as the cartoon-texture decomposition \cite{vese2004image,osher2003image}. This imposes a total variation prior on the cartoon component, and a dual Sobolev norm on the texture component. In each case, using multiple regularizers requires solving a composite optimization problem with at least three terms. We propose to leverage this generalization in addition to the aforementioned infimal convolution fidelity, allowing for more expressive reconstruction capabilities.

\subsection{Contributions}
We propose the following contributions.
\begin{itemize}
    \item For problems with two sources of noise, we propose to use the infimal convolution fidelity term within PnP algorithms. We focus on integrating this fidelity into PnP algorithms utilizing the gradient of the fidelity. When both noises are log-concave, we provide two cases where the gradient of the infimal convolution is easily computable: assuming one is differentiable, the other has to be (i) differentiable, or (ii) strictly log-concave (possibly nonsmooth) and corresponding to additive noise. These decompositions include Poisson--Gaussian and Laplacian-Gaussian noise respectively as special cases. 
    \item We further consider the generalized PnP problem corresponding to three-term variational problems, using the Davis--Yin three-operator splitting to implement two deep denoisers as priors, denoted PnP-DYS. This is a general algorithm that can use standard fidelity terms such as the squared loss, or the infimal convolution fidelity. 
    \item We show that the three-operator splitting PnP-DYS is convergent with weak assumptions, namely under the standard assumption that one denoiser is a Lipschitz-restricted gradient step denoiser, and the other is simply a single-valued proximal function. Our result complements existing results on PnP methods using extrapolated three-operator splittings \cite{wu2024extrapolated}, while imposing weaker assumptions on the denoisers. This allows us to use \textit{two} deep denoisers to regularize, instead of only one deep denoiser and one classical regularizer like Tikhonov.
    
    \item We provide detailed experiments and hyperparameter ablations for natural images with mixed noise, including effects of fidelity mismatch, multiple priors, and convergence behavior across different mixed noise scenarios. Using the infimal convolution within PnP-DYS and PnP-PGD, we demonstrate that incorporating a mixed fidelity is necessary for many provable PnP methods to empirically converge for mixed noise tasks. Moreover, the infimal convolution can provide stable reconstruction with noise up to 38\% standard deviation, significantly higher than the 5\% considered in the literature. Qualitative experiments also demonstrate reduced oversmoothing in high noise scenarios when using PnP-DYS compared to PnP-PGD.
\end{itemize}

This work is structured as follows. In \Cref{sec:problem}, we recall the infimal convolution framework, and how it can be used to address mixed noise corruption. Utilizing the gradient of the convolution allows for substitution into the PnP framework. \Cref{sec:DYS} briefly recalls the Davis--Yin three-operator splitting. We describe the PnP-DYS formulation using two denoisers, complementary to the multiple fidelity approach introduced in \cite{wu2024extrapolated}. In \Cref{sec:convergence}, we demonstrate that under standard \KLz assumptions on the denoisers, the PnP iterations with infimal convolution fidelity and multiple denoiser priors converge. The experimental \Cref{sec:experiments} provides extensive numerical experiments of both single-prior and multiple-prior PnP algorithms, detailing the convergence and performance across different hyperparameter settings and corruption operators.

\section{Infimal Convolution Plug-and-Play}\label{sec:problem}
This section introduces the infimal convolution framework of \cite{calatroni2017infimal}, motivated from a Bayesian point of view. Using a nested minimization argument to compute the gradient of the infimal convolution, we can use this to implement PnP methods involving the gradient of the fidelity term. We then illustrate how to implement the infimal convolution into the three-operator Davis--Yin splitting.

We recall the linear inverse problem with mixed noise, which we assume has one additive component. This includes noise processes such as additive Laplace--Gaussian noise, or Poisson--Gaussian noise. The forward process can be formulated as follows: for a ground truth $x \in \R^n$ and linear operator $A: \Rn \to \Rm$, an observation $y$ is obtained by 
\begin{equation}
    y = z(Ax) + \omega,
\end{equation}
where $z$ itself is a noise degradation operator with likelihood $\ell(z|Ax)$, and $\omega$ is some independent additive noise. The infimal convolution can be derived by considering the joint MAP estimator over $x$ and $\omega$,
\begin{equation}\label{eq:MAPEstimator}
    (\xhat, \hat{\omega}) = \argmax_{x, \omega} p(x, \omega | y).
\end{equation}
By Bayes' theorem, this can be rewritten in terms of the likelihood and the prior as 
\begin{align}
    (\xhat, \hat{\omega}) &= \argmax_{x, \omega} p(y|x, \omega) p(x, \omega) \notag\\
    &= \argmax_{x, \omega} \ell(y  - \omega| Ax) p(x) p(\omega). \label{eq:MAPSplit}
\end{align}
This decomposes the MAP estimation problem into a product of the noise likelihoods $\ell(y-\omega| Ax)$ and $p(\omega)$, as well as the prior $p(x)$. By taking negative logarithms, the problem \labelcref{eq:MAPSplit} can be reformulated as a variational minimization:
\begin{equation}
    (\hat x, \hat \omega) = \argmin_{x, \omega} \left\{ f_1(Ax;y-\omega ) + f_2(\omega) + g(x) \right\},
\end{equation}
where $f_1 = - \log \ell(y-\omega | Ax)$ and $f_2 = -\log p(\omega)$ are the (convex) data fidelity terms corresponding to the noising processes $z(Ax)$ and $\omega$ respectively, and $g(x) = -\log p(x)$ represents the image prior. A key insight of \cite{calatroni2017infimal} is that the joint minimization over $x, \omega$ can be factorized into a nested minimization,
\begin{equation}\label{eq:nestedExplicit}
    \min_x \left[g(x) + \min_\omega \left(f_1(Ax;y-\omega) + f_2(\omega)\right)\right].
\end{equation}
Observe that the inner minimization term is in convolution form, where the arguments of $f_1$ and $f_2$ are $y-\omega$ and $\omega$ respectively. The \textit{infimal convolution} is defined to be this minimizer
\begin{equation}\label{eq:general_phi}
    (f_1 \infconv f_2)(Ax) \coloneqq \inf_{\omega}\left\{f_1(Ax; y-\omega) + f_2(\omega) \right\}.
\end{equation}
This groups the likelihoods from both noise components into one fidelity term. The MAP estimation problem of \labelcref{eq:MAPEstimator} can therefore be written as a standard two-term composite minimization problem
\begin{equation*}
    \min_x \left[(f_1 \infconv f_2)(Ax) + g(x)\right].
\end{equation*}
For the particular case of total variation regularization $g(x) = \|\nabla x\|_1$ in \cite{calatroni2017infimal}, this minimization is done by standard primal dual methods. In order to apply more general composite minimization algorithms, we need a descent direction on both fidelity and regularization. In the PnP framework, descent on $g$ is given by a denoiser. The next section considers the gradient of the infimal convolution $\nabla (f_1 \infconv f_2)$, and how it can be computed. 

\subsection{Gradients of the infimal convolution}\label{eq:gradsInfconv}

The gradient $\nabla_x (f_1 \infconv f_2)(Ax)$ of the infimal convolution \labelcref{eq:general_phi} can be written in terms of a partial derivative of the joint objective $\phi(Ax, \omega) = f_1(Ax; y-\omega) + f_2(\omega)$. This is formalized in the following proposition, and is sometimes known as the envelope theorem.
\begin{proposition}[{\cite[Cor. 10.14]{rockafellar1998variational}}]\label{prop:diffUnderMin}
    Let $\phi:\R^n \times \R^m \rightarrow \overline \R$ be a convex proper closed function. Assume further that $\phi(x,\omega)$ is level-bounded in $\omega$ locally uniformly in $x$, i.e. for any $\alpha \in \R$, the level sets $\{\omega \mid \phi(x,\omega) \le \alpha\}$ are locally uniformly bounded. Let $p(x) \coloneqq \inf_\omega \phi(x,\omega)$ be the optimum over the second argument, and let $\bar x \in \dom p$. If the point $\bar\omega$ attaining the minimizer in $p(\bar x)$ is unique, and $\phi$ is differentiable in $x$ at $(\bar x, \bar\omega)$, then $p$ is strictly differentiable at $\bar x$ with $\nabla p(\bar x) = \nabla_x \phi(\bar x, \bar \omega)$.
\end{proposition}
There are three main assumptions that have to be verified, namely local level-boundedness, unique minimizers with respect to the slack variable, and the partial differentiability condition. In the case of $\gC^1$ log-concave likelihoods plus Gaussian noise, the first two are satisfied due to the strong convexity of the Gaussian log-likelihood. The last condition can be enforced using a change of variables, as will be demonstrated for Laplace--Gaussian noise. This implies that finding the gradient of the infimal convolution is equivalent to finding a (unique) minimizer of the convolution.

In the additive Gaussian noise case, the infimal convolution can be interpreted as a Moreau envelope evaluated at $y$. More concretely, let $f_1(Ax; y)$ be the convex negative log-likelihood of the process $z(x)$, and let $f_2(\omega) = \frac{1}{2\sigma^2} \|\omega\|^2$ be the Gaussian negative log-likelihood. We further assume that $f_1(Ax;y)$  is either differentiable in $x$, or of the residual form $f_1(Ax,y) = f_1(y-Ax)$. Applying the change of variables $v = y-\omega$, the interior term of \labelcref{eq:nestedExplicit} becomes
\begin{align}
    \psi(x) \coloneqq \min_{\omega} \left[f_1(Ax;y - \omega) + f_2(\omega)\right] &= \min_{v} \left[f_1(Ax;v)+ \frac{1}{2\sigma^2} \|y-v\|^2\right] \label{eq:modification}\\
    &= \gM_{\sigma^2 f_1(Ax, \cdot)}(y), \notag
\end{align}
where $\gM$ denotes the Moreau envelope of the function. If $f_1$ is differentiable in $x$, then by \Cref{prop:diffUnderMin}, the derivative of the infimal convolution $\psi(x) = [f_1(Ax, \cdot) \infconv f_2](y)$ is 
\begin{align*}
    \nabla_x \psi(x) = A^\top \partial_{Ax} f_1(\prox_{\sigma^2 f_1(Ax, \cdot)}(y), Ax).
\end{align*}

An example of this is Poisson--Gaussian noise, where $z(x) \sim \mathrm{Pois}(Ax)$. The fidelity term corresponding to Poisson noise is (up to some constants)
\begin{equation}\label{eq:poissonDiscreteLLhood}
    -\log p(v| Ax) = \sum_i (Ax)_i - v_i \log(Ax)_i + \log (v_i!) + \iota_{v_i (Ax)_i \ge 0}(v).
\end{equation}
where $\iota_C$ is the indicator function of a convex set $C \subset \R^n$, satisfying $\iota_C(x) = 0$ if $x \in C$, and $\iota_C(x) = +\infty$ otherwise. As in \cite{calatroni2017infimal}, we have to use Stirling's approximation $\log v! \approx v\log v - v$ to pass into continuous $v$. The fidelity used for the Poisson noise process is thus the continuous approximation
\begin{equation*}
    f_1(Ax; v) \coloneqq \KL(v, Ax) = \sum_i v_i\log \frac{v_i}{(Ax)_i} - v_i + (Ax)_i.
\end{equation*}
\begin{remark}
    Typical first-order methods addressing Poisson noise discard the $\log (v_i!)$ term in \labelcref{eq:poissonDiscreteLLhood}, as it disappears when differentiating with respect to $x$ \cite{hurault2023bregman,melidonis2023efficient}. However, since we need to find the minimizer of the convolution functional $f_1(Ax;v) + f_2(y-v)$ with respect to $v$, this can not be canceled, and needs to be justified with an approximation using Stirling's formula.
\end{remark}

For an $x$ with $Ax > 0$ componentwise, the generalized fidelity satisfies $\psi(x) = \min_v \left[\KL(v \|Ax) + \frac{1}{2\sigma^2}\|y-v\|^2\right]$. The unique $\bar v$ attaining the minimizer satisfies componentwise
\begin{equation}\label{eq:internalOpti}
    \log \frac{\bar{v}_i}{(Ax)_i}+ \frac{1}{\sigma^2}(\bar v-y)_i = 0,
\end{equation}
and the gradient of the fidelity $\nabla \psi(x)$ is given by 
\begin{equation}\label{eq:PoissonGaussianGrad}
    \nabla_x\left[ \KL(\bar v \| Ax) + \frac{1}{2\sigma^2}\|y-\bar v\|^2\right] = A^\top \left(\mathbf{1} - \left[\frac{\bar v_i}{(Ax)_i}\right]_i\right).
\end{equation}
The gradient of the Poisson--Gaussian fidelity term depends on computing the optimal slack variable $\bar v$ in \labelcref{eq:internalOpti}. In this case, the separability turns this into a collection of 1D root finding problems, which can be very efficiently vectorized and solved using standard methods such as Newton's method. Adding regularization parameters $(\lambda_p, \lambda_g)$ for the Poisson and Gaussian components yields the following expressions for the infimal convolution of Poisson--Gaussian noise,
\begin{subequations}\label{eqs:PoissonGaussianGrad}
    \begin{gather}
        \psi(x) = \min_v \left[\lambda_p\KL(v \|Ax) + \frac{\lambda_g}{2}\|y-v\|^2\right],\\
        \lambda_p\log \frac{\bar{v}_i}{(Ax)_i}+ \lambda_g(\bar v-y)_i = 0,\\
        \nabla\psi(x)  = \lambda_p A^\top \left( \mathbf{1} - \left[\frac{\bar v_i}{(Ax)_i}\right]_i\right).
    \end{gather}
\end{subequations}

The other scenario arises if $f_1$ is of the residual form $f_1(Ax;y) = f_1(y-Ax)$, and is convex but not necessarily differentiable. The residual form is equivalent to $z(x)$ being an additive noise process. In this case, one has to move the $Ax$ term onto the differentiable $f_2$ term. The infimal convolution fidelity term \labelcref{eq:modification} can equivalently be written as 
\begin{equation*}
    \psi(x) = \min_{w} \left[f_1(w) + \frac{1}{2\sigma^2} \|y-Ax-w\|^2\right].
\end{equation*}
This reformulation allows for use of \Cref{prop:diffUnderMin}, as the objective in this minimization is differentiable in $x$, and furthermore admits a unique minimizer $w$ for any fixed $x$ (due to strict convexity of the square term). An example is Laplace--Gaussian noise, where $z(x) = Ax + w$ where $p(w) \propto \exp(-\lambda_l\|w\|_1)$ is Laplace-distributed. The fidelities can be given by 
\begin{equation*}
    f_1(Ax,z) = \lambda_l\|z-Ax\|_1,\quad f_2(\omega) = \frac{\lambda_g}{2}\|\omega\|_2^2,
\end{equation*}
where $\lambda_l, \lambda_g>0$ are some scale parameters, which can be tuned for PnP. Then, the infimal convolution is simply
\begin{equation}\label{eq:LaplaceGaussianFidelity}
    \psi(x) = \min_w \left[\lambda_l\|w\|_1 + \frac{\lambda_g}{2}\|y-Ax-w\|^2\right].
\end{equation}
The minimizer is attained at the easily computable $\ell_1$-proximal
\begin{equation*}
    \bar w = \prox_{\frac{\lambda_l}{\lambda_g}\|\cdot \|_1}(y-Ax),
\end{equation*}
with the gradient of the infimal convolution following
\begin{equation}\label{eq:LaplaceGaussianFidelityGrad}
    \nabla_x \psi(x) = \lambda_g A^\top \left(Ax + \bar w - y\right).
\end{equation}

This discussion shows that for the case where $f_2$ corresponds to Gaussian additive noise, the infimal convolution fidelity is differentiable if (i) $f_1$ is differentiable in $x$, or (ii) takes the residual form $f_1(y, Ax) = f_1(y-Ax)$ arising from additive noise. Poisson--Gaussian and Laplace--Gaussian mixed noise are representative examples of each condition, with particularly simple gradients of the infimal convolution, with the former only requiring some additional 1D root finding. In principle, it is possible to extend this computation to other combinations of log-concave noise satisfying suitable regularity conditions, such that the assumptions of \Cref{prop:diffUnderMin} hold. 

At this point, the infimal convolution fidelity can be directly applied into PnP-like schemes utilizing one denoiser. In the next section, we detail how using two regularization terms can be done using a three-operator splitting.

\section{Davis--Yin splitting and multiple denoiser priors}\label{sec:DYS}
Two-operator splittings are used to target composite optimization problems of the form $\min( f+g)$ with one regularization $g$. Using two regularizations $g,h$ means to solve a three-term composite optimization problem of the form $\min (f+g+h)$. The Davis--Yin splitting does precisely this: for proper closed convex functions $f,g,h$ with $f$ being differentiable with $L$-Lipschitz gradient, the minimization problem can be solved using a fixed point iteration on the operator
\begin{equation}\label{eq:DYSOperator}
    \gT \coloneqq \Id - \prox_{\gamma h} + \prox_{\gamma g} \circ\ (2 \prox_{\gamma h} - \Id - \gamma \nabla f \circ \prox_{\gamma h}),
\end{equation}
where $\gamma \in (0, 2/L)$ is a step size. A standard approach is to use a relaxation parameter $\eta \in (0,1]$, and perform the Krasnosel'ski\u{\i}--Mann iteration
\begin{equation}\label{eq:KrasnoselskiiMann}
    w^{k+1} = (1-\eta) w^{k} + \eta \gT w^k.
\end{equation}
It can be shown that under suitable conditions on $\gamma$ and $\eta$, the fixed point iteration $z_{k+1} = \gT(z_k)$ converges to an element $z_*$ satisfying $\prox_{\gamma h} (z_*) \in \zer \partial (f+g+h)$. The concept of using multiple priors has been previously explored in works such as \cite{rond2016poisson}, which employed an alternating scheme for Poisson denoising, albeit without convergence guarantees. 

Following the PnP framework, we substitute the proximal operators $\prox_{\gamma g}$ and $\prox_{\gamma h}$ with suitable pretrained denoisers $\gD_{1}$ and $\gD_{2}$, respectively. This leads to an iteration scheme that can handle two regularization terms and one smooth data-fidelity term simultaneously. For a differentiable fidelity function $\psi(x)$, the \textit{PnP-DYS} method with step size $\gamma < 2/\operatorname{Lip}(\nabla \psi)$ and relaxation $\eta \in (0,1]$ is given by the iterations
\begin{equation}\label{eq:pnp-dys}
\begin{cases}
    x^{k+1} = \gD_{1}(w^k), \\
    u^{k+1} = \gD_{2}(2x^{k+1}-w^k-\gamma\nabla \psi(x^{k+1})), \\
    w^{k+1} = w^k + \eta (u^{k+1}-x^{k+1}).
\end{cases}
\end{equation}
In addition to evaluating the denoiser, this additionally requires the gradient of the fidelity term $\psi$. For typical PnP applications, the fidelity is simply the squared $\ell_2$ distance, which has a simple gradient assuming $A^\top$ is easy to compute. For infimal convolution fidelities arising from mixed-noise processes, derivatives can be computed according to \Cref{prop:diffUnderMin}. We present the PnP-DYS method specialized for Poisson--Gaussian noise and Laplace--Gaussian noise in Algorithm \ref{alg:pnp-dys}.

\begin{algorithm}[H]
\renewcommand{\algorithmicrequire}{\textbf{Input:}}
\renewcommand{\algorithmicensure}{\textbf{Output:}}
\caption{Infimal Convolution Plug-and-Play Davis--Yin Splitting (PnP-DYS)}
\label{alg:pnp-dys}
\begin{algorithmic}[1]
\REQUIRE Corrupted image $y$, linear forward operator $A$, step size $\gamma > 0$, relaxation $\eta \in (0, 1]$, regularization parameters $\lambda_g$, $\lambda_l$ (Laplace) or $\lambda_p$ (Poisson), infimal convolution fidelity $\psi$, iteration count $K_{\max}$, pretrained denoisers $\gD_1, \gD_2$.
\ENSURE Restored image
\STATE Initialize $w^0$
\FOR{$k = 0, 1, 2, \dots, K_{\max}-1$}
    \STATE $x^{k+1} = \gD_{1}(w^k)$
    \STATE \COMMENT{\textit{$\vartriangleright$ Compute gradient of infimal convolution $\nabla \psi(x^{k+1})$}}
    \IF{Laplace--Gaussian noise}
        \STATE \[
        \nabla\psi(x^{k+1}) = -\lamg A^\top\left(\Id - \prox_{\frac{\laml}{\lamg}\norm{\cdot}_1}\right)(y - Ax^{k+1})
        \]
    \ELSIF{Poisson--Gaussian noise} 
    \STATE Initialize ${ v}^0 = Ax^{k+1}$
    \FOR{$m = 0,1,\dots,M-1$} 
        \STATE ${v}^{m+1}_i = {v}^m_i - \frac{\lamp \log({ v}^m_i/[Ax^{k+1}]_i) + \lamg({v}^m_i - y_i)}{\lamp/{v}^m_i + \lamg}$, $\forall i$ \hfill\COMMENT{\textit{$\vartriangleright$ Newton iteration}}
    \ENDFOR
    \STATE \[
    \nabla \psi(x^{k+1}) = \lamp A^\top \left( \mathbf{1} - \frac{{v^M}(x^{k+1})}{Ax^{k+1}} \right)
    \]
    \ENDIF
    \STATE $u^{k+1} = \gD_{2}\left( 2x^{k+1} - w^k - \gamma \nabla \psi(x^{k+1}) \right)$
    \STATE $w^{k+1} = w^k + \eta(u^{k+1} - x^{k+1})$
\ENDFOR
\STATE return $x^K$ 
\end{algorithmic}
\end{algorithm}

\section{Convergence Analysis}\label{sec:convergence}
To analyze the fixed-point convergence of PnP methods, a common strategy is to use a particular proximal characterization of the gradient-step denoiser, relaxing the convex composite optimization problem into a weakly convex problem. This is then followed by using a \KLz (KL) condition to conclude convergence to a stationary point of some functional \cite{hurault2022proximal,gribonval2020characterization,tan2025solving}. In this section, we demonstrate that a similar convergence holds for PnP-DYS methods, under standard conditions on the parameters of the Davis--Yin splitting. Furthermore, we show that the proposed infimal convolutions satisfy the assumptions for fixed-point convergence. This gives a characterization of the cluster points in terms of a non-convex functional, implicitly defined by the denoisers.

We consider two gradient step denoisers $\gD_{i}$ for $i=1,2$ \cite{hurault2021gradient}. These are defined by the gradient formulation
\begin{equation}\label{eq:FormGSDenoiser}
\gD_{i} = \Id - \nabla g_{i},\quad g_{i}(x) = \frac{1}{2}\norm{x-N_{i}(x)}^2,\quad i=1,2,
\end{equation}
where $g_{i}:\Rn\rightarrow \R$ are $\gC^2$ functions with $L$-Lipschitz gradient and $L<1$, and $N_{i}:\Rn\to\Rn$ are parameterized by $\gC^2$ neural networks, enforceable through having a $\gC^2$ activation function. Under the Lipschitz condition on $\nabla g_{i}$, the denoisers $\gD_{i}$ can be expressed as proximal mappings of a weakly convex function, as described in the following proposition.

\begin{proposition}[{\cite[Prop. 1]{hurault2023relaxed}, \cite[Thm. 1]{gribonval2020characterization}}]\label{prop:weakly-convex}
Suppose that $\gD = \Id - \nabla g$, where $g:\Rn\rightarrow \R$ is a $\gC^2$ function, with $\nabla g$ being $L_g$-Lipschitz with $L_g<1$. Then  $\gD$ takes the proximal form $\gD=\prox_{\phi}$, where
\begin{equation}\label{eq:phi-function}
 \phi=
\begin{cases}
    g(\gD^{-1}(x))-\frac{1}{2}\norm{\gD^{-1}(x)-x}^2, &\text{if } x\in \Im(\gD), \\
   +\infty, &\text{otherwise.}
\end{cases}
\end{equation}
where $\phi$ is a $\frac{L_g}{L_g+1}$-weakly convex function. Moreover, $\phi$ is differentiable on $\Im(\gD)$, and its gradient $\nabla \phi$ is $\frac{L_g}{1-L_g}$-Lipschitz.
\end{proposition}
To relate this proximal characterization to the PnP-DYS iterations \labelcref{eq:pnp-dys}, it is natural to consider the non-convex objective
\begin{equation}\label{eq:proxyF}
    F(x) \coloneqq \psi(x) + \frac{1}{\gamma} \phi_{1}(x)+ \frac{1}{\gamma} \phi_2(x).
\end{equation}
Formally applying the DYS iterations \labelcref{eq:DYSOperator,eq:KrasnoselskiiMann} to this objective \labelcref{eq:proxyF} with step size $\gamma$ yields the PnP-DYS iteration \labelcref{eq:pnp-dys}. Since the $\phi_i$ defined through the denoiser will be weakly convex functions, convergence to a minimizer is not guaranteed. A natural weaker objective is to check whether or not the PnP-DYS iterations converge to critical points of $F$, which can be done using the KL condition. For a nonconvex function, the KL property converts a sufficient decrease condition of a sequence into convergence. Moreover, the convergence is fast in the sense of having finite length.   

\begin{definition}[\KLz property]\label{def:KL}
    We say that a proper closed function $f$ satisfies the \KLz property at $\xsol\in \dom\partial f$ if there exists a neighborhood $V$ of $\xsol$, $\delta\in(0,\infty)$ and a continuous concave function $\Psi:[0,\delta)\rightarrow \sR_+$ with $\Psi(0)=0$ such that
\begin{enumerate}[label={\rm (\roman{*})}]
 \item $\Psi$ is continuous on $[0,\delta)$, and continuously differentiable on $(0,\delta)$ with $\Psi'>0$;
 \item For all $x\in V$ with $f(\xsol)<f(x)<f(\xsol)+\delta$, one has
 \[
 \Psi'(f(x)-f(\xsol))\dist(0,\partial f)\geq 1.
 \]
\end{enumerate}
A proper closed function $f$ satisfying the \KLz property at all points in $\dom\partial f$ is called a \KLz function.
\end{definition}
In the context of PnP algorithms, neural networks can be parameterized to satisfy the KL condition, such as using networks with piecewise polynomial activations, or analytic activations such as Softplus. For other sufficient conditions to guarantee the KL condition, we refer to \cite[Sec. 4.2]{tan2025solving}.

In order to show convergence to a fixed point, we first need to show a descent condition on the PnP-DYS iterates. This is given by the following proposition. Proofs of following results are given in \Cref{sec:proofs}.

\begin{proposition}\label{prop:descent}
Consider applying the relaxed Davis--Yin splitting to the nonconvex function \labelcref{eq:proxyF} with step size $\gamma>0$ and relaxation $\eta \in (0,1]$, i.e. the iterations
\begin{equation}\label{maineq:pnp-dys-prox}
\begin{cases}
    x^{k+1} = \prox_{\phi_1}(w^k), \\
    u^{k+1} = \prox_{\phi_2}(2x^{k+1}-w^k-\gamma\nabla \psi(x^{k+1})), \\
    w^{k+1} = w^k + \eta (u^{k+1}-x^{k+1}).
\end{cases}
\end{equation}
Assume the following conditions: 
\begin{enumerate}
    \item $\nabla\psi$ is $\beta$-Lipschitz,
    \item $\phi_1$ is $l_{\phi}$-weakly convex with $l_{\phi} \in [0,1)$, and moreover has $L_{\phi}$-Lipschitz gradient,
    \item $\phi_2$ is proper and closed with $\prox_{\phi_2}$ being nonempty.
\end{enumerate} 

Define the PnP-DYS energy function associated with
\eqref{eq:proxyF} by
\beq\label{eq:energy_func}
\begin{aligned}
\Theta_{\gamma,\eta}(x,u,w)
\coloneqq &\;
\psi(x) + \frac{1}{\gamma}\phi_1(x)
+
\frac{1}{\gamma}\phi_2(u) +
\frac{1}{2\gamma}
\norm{2x-u-w-\gamma\nabla\psi(x)}^2\\
&-
\frac{1}{2\gamma}
\norm{w-x+\gamma\nabla\psi(x)}^2
-
\frac{\eta}{\gamma}\norm{x-u}^2 .
\end{aligned}
\eeq
Further define the following descent constant
% \beq
\begin{equation}\label{eq:Gamma}
    \Lambda(\gamma, \eta) \coloneqq \frac{1}{\gamma} \left[-\frac{1 + l_\phi}{2} + \frac{1 - L_\phi^2}{\eta}\right] - \beta \left[1 + \frac{(1 - \eta + L_\phi)^2}{2\eta^2}\right].
\end{equation}
If $\Lambda(\gamma, \eta)>0$, then for all $k\geq 1$, the descent on $\Theta_{\gamma,\eta}$ holds
\begin{equation*}
\Theta_{\gamma,\eta}(x^{k+1},u^{k+1},w^{k+1})
-
\Theta_{\gamma,\eta}(x^k,u^k,w^k)
\leq
-\Lambda(\gamma,\eta)\norm{x^{k+1}-x^k}^2 .
\end{equation*}
\end{proposition}

A consequence of the descent lemma is convergence to a cluster point, as given by the KL condition.

\begin{theorem}\label{thm:global_convergence}
Let $\{(x^k,u^k,w^k)\}_{k\geq 0}$ be the sequence generated by the
PnP-DYS iteration \eqref{maineq:pnp-dys-prox}. Suppose that the
assumptions of \Cref{prop:descent} hold and that $\gamma>0$ and
$\eta\in(0,1]$ are chosen such that $\Lambda(\gamma,\eta)>0$.

Assume further that the sequence $\{(x^k,u^k,w^k)\}_{k\geq 0}$ is bounded, and $\Theta_{\gamma,\eta}$ is a KL function, then the sequence $\{(x^k,u^k,w^k)\}_{k\geq0}$ has finite length, that is,
    \[
    \sum_{k=0}^{\infty}\norm{x^{k+1}-x^k}<+\infty,
    \qquad
    \sum_{k=0}^{\infty}\norm{u^{k+1}-u^k}<+\infty,
    \qquad
    \sum_{k=0}^{\infty}\norm{w^{k+1}-w^k}<+\infty.
    \]
In particular, the sequences $(x^k)$ and $(u^k)$ converge to a critical point of $F = \psi + \gamma^{-1} \phi_1 + \gamma^{-1} \phi_2$.
\end{theorem}

\Cref{prop:descent} generalizes the analysis of \cite{bian2021three} to the relaxed DYS iterations, namely recovering the unrelaxed Lyapunov function and step size condition on $\gamma$ by setting $\eta = 1$. The different assumptions on the three terms inform the replacement within the PnP framework, with \cite{wu2024extrapolated} replacing the $\prox_{\phi_2}$ term with a deep denoiser to always guarantee convergence for sufficiently small step size, while maintaining an easily computable $\prox_{\phi_1}$ using Tikhonov regularization. Another minor theoretical difference is that \cite[Thm. 3.10]{wu2024extrapolated} imposes a twice-differentiability with bounded Hessian assumption on the fidelity function, which is not strictly necessary in our proof. 

In order to replace both $\phi_1$ and $\phi_2$ with deep denoisers as we propose, additional constraints on the Lipschitz constants of the denoisers must be placed. For feasible parameters to exist, namely having $\gamma, \eta$ such that $\Lambda(\gamma, \eta)>0$, a necessary and sufficient condition is that $1-L_\phi^2 > 0$. This can be seen by taking $\eta,\beta\gamma \rightarrow 0$. Recall from \Cref{prop:weakly-convex}, if $\phi_1$ is implicitly defined by a gradient step denoiser $\gD = \Id - \nabla g$ where $\nabla g$ is $L_g$-Lipschitz, then the Lipschitz constant of $\nabla \phi$ is bounded by $\frac{L_g}{1-L_g}$. A sufficient condition is therefore $L_g < 1/2$. The Lipschitz constant $\beta$ of the fidelity's gradient can be controlled by simply multiplying the fidelity function by an appropriate regularization parameter, and gives an upper bound on the amount of regularization available similarly to other provable PnP methods \cite{tan2025solving}. 

This restriction that $L_g < 1/2$ is precisely the same as the analysis for the Douglas--Rachford splitting when the fidelity term is nondifferentiable \cite[Thm. 4.4]{hurault2022proximal}. This can be addressed by using a relaxed denoiser of the form $D_{\sigma,\alpha} = I - \nabla g_\alpha = I - \alpha \nabla g$, $\alpha \in (0,1)$, and applying \Cref{prop:weakly-convex} with $\nabla g_\alpha$ being $\alpha L_g$-Lipschitz.

The boundedness assumption is commonly used in nonconvex KL-based convergence analysis. In general, the descent property of $\Theta_{\gamma,\eta}$ alone does not automatically imply boundedness of the sequence. However, this assumption can be enforced by using a clipping step after each iteration such as projecting onto $x\in [0,1]^n$.

The assumptions in Proposition \ref{prop:descent} are satisfied automatically by the
Laplace--Gaussian fidelity since it is a Moreau envelope with
globally Lipschitz gradient. In particular,
\[
    \beta=\lambda_g\lambda_{\max}(A^\top A).
\]
where $\lambda_g$ denotes the Gaussian fidelity parameter introduced in \eqref{eq:LaplaceGaussianFidelity}. For the Poisson--Gaussian fidelity, the same condition holds only locally, for instance when $Ax\geq\varepsilon\mathbf{1}$ componentwise for some
$\varepsilon>0$. If the generated sequence is eventually within a compact set in $\R^n_+$, then the assumptions of Proposition \ref{prop:descent} are satisfied.

\begin{remark}[KL property of the fidelity terms]\label{rem:KL-fidelity}
Note that for the Laplace--Gaussian model, the data-fidelity term is given by the infimal convolution of an $\ell_1$ term and a quadratic term, which yields a Huber-type piecewise quadratic function composed with the linear operator $A$. Hence, it is semialgebraic and therefore satisfies the KL property.

For the Poisson--Gaussian model, the data-fidelity term is defined on the positive domain. By the analysis in \cite{daniele2026deep}, such fidelity terms can be handled locally away from the boundary of the domain. In particular, if the generated iterates remain in a compact subset of
\[
S=\{x\in\mathbb{R}^n: Ax\geq \varepsilon\mathbf{1}\}
\]
for some $\varepsilon>0$, then the fidelity is analytic and therefore satisfies the KL property on $S$.
\end{remark}

\section{Experiments}\label{sec:experiments}
In this section, we evaluate the performance of the proposed infimal convolution fidelity, with and without the two-denoiser formulation of PnP-DYS. We demonstrate that the infimal convolution effectively models the forward process, and that using multiple priors can improve qualitative performance compared to single denoiser PnP methods. We consider image deblurring and super-resolution tasks, under Laplace--Gaussian noise in \Cref{ssec:LaplaceGaussian}, and Poisson--Gaussian noise in \Cref{ssec:PoissonGaussian}. Moreover, we consider reconstruction in two high noise settings, $10\%$ plus $10\%$ Laplace--Gaussian noise, and $27\%$ plus $3\%$ Laplace--Gaussian noise. The high noise scenarios demonstrate that integrating classical priors can help to stabilize reconstruction outside of the well-trained regions of a deep denoiser.

For the pre-trained denoisers $\gD_1, \gD_2$ employed in our PnP-DYS algorithm, we consider two distinct types of denoising operators. The first is the pretrained gradient-step DRUNet \labelcref{eq:FormGSDenoiser}, finetuned to have the Lipschitz constant of $g$ less than 1 \cite{hurault2022proximal}. The second is a total variation (TV) denoiser \cite{rudin1992nonlinear}, given by the proximal operator with step size $\sigma$, 
$\gD_{\sigma}(x) = \prox_{\sigma \TV}(x) = \argmin_u \left\{ \|\nabla u\|_1 + \frac{1}{2\sigma}\norm{u - x}^2  \right\}$. The latter proximal is implemented using DeepInverse \cite{tachella2025deepinverse}, and is used in the extrapolated PnP-DYS algorithm of \cite{wu2024extrapolated}.

We consider two choices within the PnP-DYS framework: \dysgs, using GS-DRUNet for both denoisers with possibly different denoising strengths, and \dystv, using TV as the second denoiser. The latter is designed for large noise scenarios beyond the network's effective range, where TV offers stable performance without relying on learned priors.

We compare our proposed PnP-DYS with several representative provably convergent PnP methods: PnP-PGD, PnP-DRS, GS-PnP, as well as the non-convergent DPIR method. We use the same pretrained gradient step denoiser of \cite{hurault2022proximal} for all provably convergent methods including PnP-DYS, which has been fine-tuned to softly enforce the Lipschitz constraint $L_g<1$. For DPIR, we use the pre-trained DRUNet architecture of \cite{zhang2021plug}. All methods are run for $K_{\max}=400$ iterations unless otherwise stated. As previously found in \cite{tan2023provably}, this is sufficient for these PnP methods to converge or diverge. The first two methods are given as follows: for a fidelity function $f$, single denoiser $\gD$, and step size $\gamma>0$,
\begin{align}
     &\begin{cases}
         z^{(k+1)} = x^{(k)} - \gamma \nabla f(x^{(k)}),\\
         x^{(k+1)} = \gD(z^{(k+1)});
     \end{cases} \tag{PnP-PGD}\\
     &\begin{cases}
         y^{(k+1)} = \gD(x^{(k)}),\\ 
         z^{(k+1)} = \prox_{\gamma f} (2y^{(k+1)} - x^{(k)}),\\
         x^{(k+1)} = x^{(k)} + z^{(k+1)} - y^{(k+1)}.
     \end{cases} \tag{PnP-DRS}
\end{align}
The latter two methods have slight modifications. GS-PnP has an additional backtracking line search with an explicit Lyapunov function, while DPIR has a decreasing step size schedule, log-spaced from a maximum strength of $\sigma_1 = 49$ to a variable $\sigma_{N}$. Additional details are provided in \Cref{app:PnPMethods}.

The choice of descent operator of the PnP methods also varies between $\nabla f$ and $\prox_f$. While the gradient can be computed for the infimal convolution fidelity $\psi_{\mathrm{IC}}$, the proximal operator does not admit an easily computable form. This precludes using the infimal convolution fidelity in methods such as GS-PnP, PnP-DRS, and DPIR, which require computing $\prox_f$. For these methods, we use the Gaussian fidelity $\psi_g(x) = {\lambda_g} \|Ax-y\|^2/2$, whose proximal is computable using Fourier transforms. Since PnP-DRS and PnP-PGD have very similar performance for the same fidelity term \cite{hurault2023relaxed}, the performance between these two reported methods is due to the choice of fidelity $\psi_{\mathrm{IC}}$ versus $\psi_g$. The methods can be grouped as follows:
\begin{enumerate}
    \item Infimal convolution fidelity, two denoisers: \dysgs, \dystv.
    \item Infimal convolution fidelity, one denoiser: PnP-PGD.
    \item Gaussian fidelity (using proximals), one denoiser: GS-PnP, PnP-DRS, DPIR.
\end{enumerate}

Hyperparameters for all methods reported are chosen by grid search over the smaller set3c dataset to maximize PSNR, including the denoiser strength(s), step size, and weighting for fidelity parameters. The detailed search grids and selected hyperparameters are reported in \Cref{sec:params}. We further conduct an ablation study to examine the sensitivity of PnP-DYS to different hyperparameter choices, and the corresponding results are summarized in \Cref{sec:ablation}.

\subsection{Laplace--Gaussian mixed noise}\label{ssec:LaplaceGaussian}
We first evaluate the performance for image deblurring and super-resolution under additive Laplace--Gaussian mixed noise in \Cref{ssec:LGDeblur,ssec:LGSR} respectively. We first report smaller scales and standard deviations corresponding to the low 1-5\% noise typically reported in PnP works, later pushing to higher noise levels to test the robustness of the proposed PnP-DYS splitting in \Cref{ssec:LGHighNoise}. Furthermore, we demonstrate that once we leave the low noise setting, using the infimal convolution fidelity over a Gaussian fidelity term is necessary in order to obtain reasonable reconstructions. 

The observed image $y$ is modeled as additive Laplace noise $\xi$ with scale parameter $\sigma_l$, with independent additive Gaussian noise $\omega$ with standard deviation $\sigma_g$,
\begin{equation}
    y = Ax + \xi + \omega, \quad \xi\sim \text{Lap}(0, \sigma_l),\; \omega \sim \gN(0, \sigma_g^2 I).
\end{equation}
The infimal convolution fidelity function and its gradient are given by \labelcref{eq:LaplaceGaussianFidelity,eq:LaplaceGaussianFidelityGrad}, where the transpose of a convolution operator can be computed using a Fourier transform. 

\subsubsection{Deblurring}\label{ssec:LGDeblur}
For the image deblurring task, the forward operator $A$ is a convolution with circular boundary conditions. We follow the setup of \cite{hurault2021gradient,hurault2023relaxed}, using a set of 10 blur kernels, visualized in \Cref{fig:all_kernels}. The blur operators are further scaled to have $\|A^\top A\|_{\mathrm{op}} \approx 0.96$. The noise standard deviation and scales are taken to be $\sigma_g, \sigma_l \in \{2.55, 7.65\}$. Since the variance of the Laplace noise is $2\sigma_l^2$, these parameters correspond to componentwise standard deviation varying from $1.7\%$ to $5.2\%$.

We report the average PSNR over the 10 blur kernels in \Cref{tab:PSNR_deblurring_LG}, and leave more detailed per-kernel averages to Appendix~\ref{sec:ex-results} in \Cref{tab:performance_final_pro}. Firstly, using an incorrect fidelity is detrimental to the stability of provably convergent methods. Even when tuned, PnP-DRS and DPIR diverge as they do not use the correct fidelity. GS-PnP is still able to achieve reasonable results due to its backtracking procedure, which produces very small step sizes near the end of the iterations. Within the methods using $\psi_{\mathrm{IC}}$, PnP-DYS achieves slightly higher PSNR compared to PnP-PGD, due to the additional flexibility endowed by the second denoiser. The two variants \dysgs and \dystv show comparable performance across different blur kernels and noise levels, with \dysgs being slightly better in this setup.

\begin{table}[]
    \centering
    \caption{Average PSNR across the 10 blur kernels over CBSD10 for different levels of Laplace--Gaussian noise. Hyperparameters are tuned on set3c. We observe that PnP-DYS is able to achieve slightly higher PSNR than the single-denoiser counterparts PnP-PGD and GS-PnP. Moreover, the methods with mismatched fidelities $\psi_g$ instead of $\psi_{\mathrm{IC}}$ have generally poorer results or diverge.}
    \label{tab:PSNR_deblurring_LG}
    \begin{tabular}{ccccccc}
         \toprule
        \multirow{3}{*}{Methods}  & \multicolumn{4}{c}{($\sigma_g, \sigma_l$)} & \multirow{3}{*}{Fidelity, \# denoisers}  \\
        \cmidrule(lr){2-5}
        & (2.55, 2.55) & (2.55, 7.65) & (7.65, 2.55) & (7.65, 7.65) & \\
        \midrule
        \dysgs & \textbf{30.79} & \textbf{28.16} & \textbf{28.87} & \textbf{27.69} & $\psi_{\mathrm{IC}},\, 2$ \\ \dystv & \underline{30.75} & \underline{28.12} & \underline{28.83} & \underline{27.62} & $\psi_{\mathrm{IC}},\, 2$ \\ PnP-PGD & 30.20 & 27.97 & 28.43 & 27.47 & $\psi_{\mathrm{IC}},\, 1$ \\ GS-PnP & 30.29 & 27.67 & 28.30 & 27.03 & $\psi_{g},\, 1$ \\ PnP-DRS & 21.37 & 16.56 & 18.05 & 15.54 & $\psi_g,\,1$ \\ DPIR & 17.83 & 12.32 & 14.05 & 11.29 & $\psi_g,\, 1$ \\
        \bottomrule
    \end{tabular}
\end{table}

A visual comparison is provided in \Cref{fig:comp_deblur}.  GS-PnP and the methods using $\psi_{\mathrm{IC}}$ all converge to similar reconstructions, but GS-PnP has more graininess within the smooth textured components. We can also observe some edge artifacts with the single-denoiser methods that are not present with two denoisers. PnP-DRS and DPIR lead to poor reconstructions and diverge respectively due to the mismatched fidelity; even at the peak PSNR taken at earlier iterations, the reported images are still corrupted.

Iteration residuals and PSNR curves are reported in \Cref{fig:comp_deblur_conv,fig:comp_deblur_psnr} for the smaller set3c dataset. For PnP-DYS, PnP-PGD and GS-PnP, the residuals converge steadily and the corresponding PSNR values increase to stable high levels, indicating stable reconstruction behavior under the mixed-noise model. GS-PnP has a stepping phenomenon at later iterations due to backtracking, where it terminates early. In contrast, although the residuals of DPIR and PnP-DRS also decrease, their PSNR curves rapidly peak then degrade, indicating convergence to a poor reconstruction.

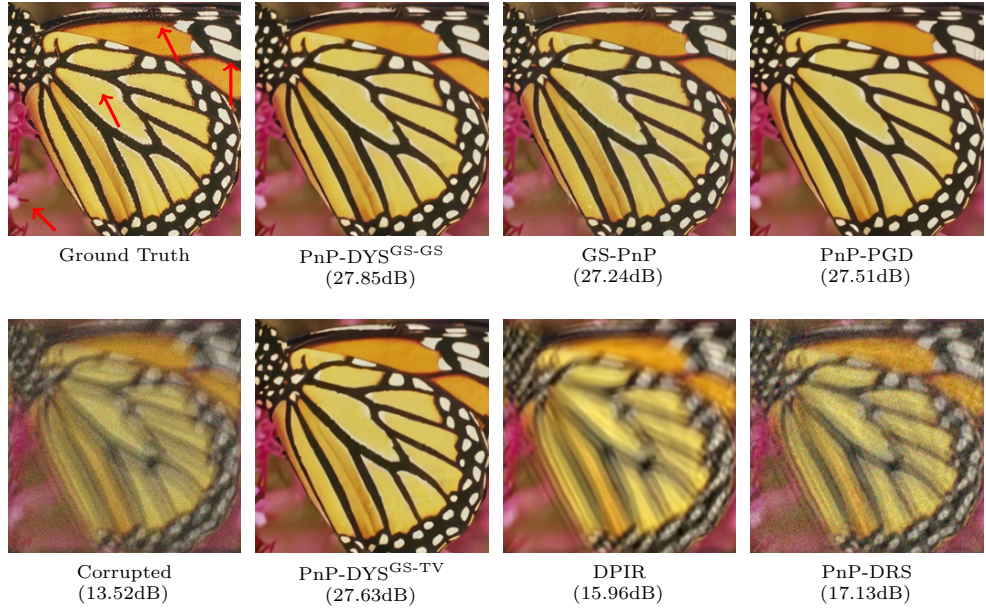
\begin{figure}[!hbtp]
    \centering \captionsetup[subfigure]{labelformat=empty,justification=centering}
    \subfloat[Ground Truth]{\tikzmarknode{bfgt}{\includegraphics[width=0.235\textwidth]{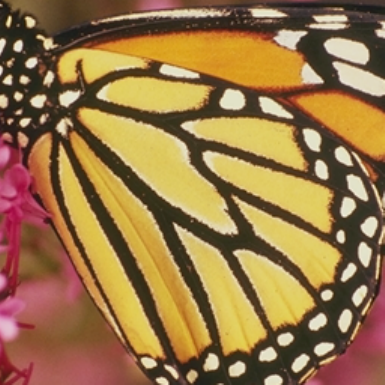}}}
    \hskip2mm
    \subfloat[\dysgs (27.85dB)][\dysgs \\(27.85dB)]{\includegraphics[width=0.235\textwidth]{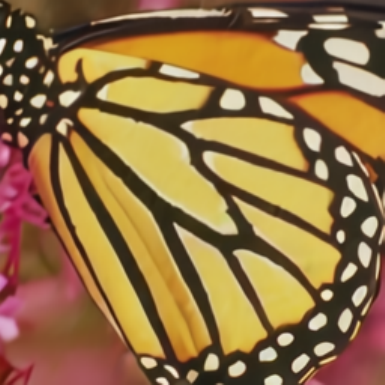}}
    \hskip2mm 
    \subfloat[GS-PnP \\ (27.24dB)]{\tikzmarknode{bfgspnp}{\includegraphics[width=0.235\textwidth]{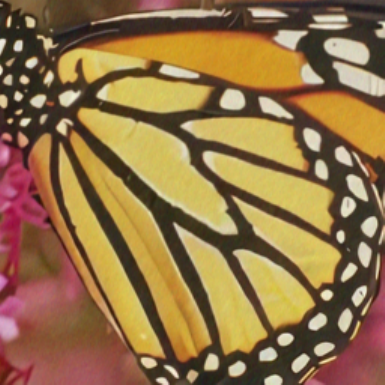}}}
    \hskip2mm
    \subfloat[PnP-PGD\\ (27.51dB)]{\tikzmarknode{bfpgd}{\includegraphics[width=0.235\textwidth]{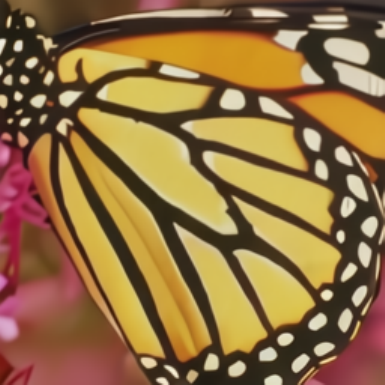}}}
    \\
    \subfloat[Corrupted \\ (13.52dB)]{\includegraphics[width=0.235\textwidth]{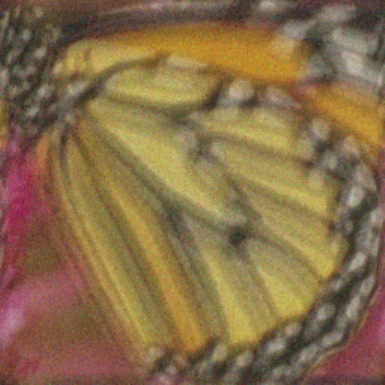}}
    \hskip2mm
    \subfloat[\dystv (27.63dB)][\dystv \\(27.63dB)]{\includegraphics[width=0.235\textwidth]{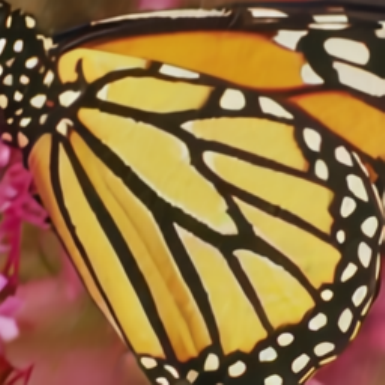}}
    \hskip2mm 
    \subfloat[DPIR \\ (15.96dB)]{\includegraphics[width=0.235\textwidth]{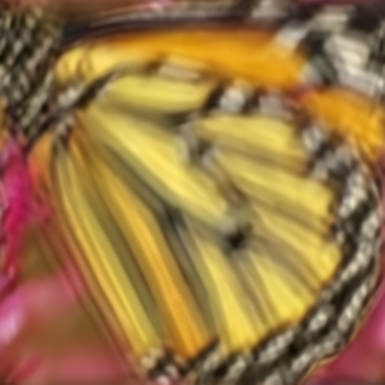}}
    \hskip2mm
    \subfloat[PnP-DRS \\ (17.13dB)]{\includegraphics[width=0.235\textwidth]{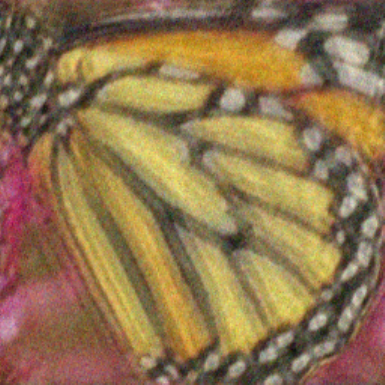}}
    \\
    \begin{tikzpicture}[remember picture,overlay]
    \node at ($(bfgt.center)+(40pt,25pt)$) (gsc) {};
    \draw[->,red,very thick] (gsc)++(0pt,-20pt) to (gsc);
    \node at ($(bfgt.center)+(-38pt,-30pt)$) (bfa) {};
    \draw[->,red,very thick] (bfa)++(12pt,-12pt) to (bfa);
    \node at ($(bfgt.center)+(12pt,38pt)$) (bfb) {};
    \draw[->,red,very thick] (bfb)++(8pt,-16pt) to (bfb);
    \node at ($(bfgt.center)+(-10pt,13pt)$) (bfc) {};
    \draw[->,red,very thick] (bfc)++(8pt,-16pt) to (bfc);
    \end{tikzpicture}
    \caption{Comparison of image deblurring results on the butterfly image with different PnP methods under blur kernel $8$ and noise levels $\sigma_g=7.65$, $\sigma_l=7.65$. For DPIR and PnP-DRS, we report the results at iteration $8$, where the PSNR values are near their early peaks before the subsequent degradation. Red arrows denote locations where artifacts or loss of detail occur on at least one of the reconstructions, mainly near the edges. We observe that PnP-PGD tends to oversmooth and GS-PnP can hallucinate edges, while the proposed PnP-DYS methods provide a reasonable smoothing. }    
    \label{fig:comp_deblur}
\end{figure}

\begin{figure}[!hbtp]
    \centering
    \subfloat[\dysgs]{\includegraphics[width=0.31\textwidth]{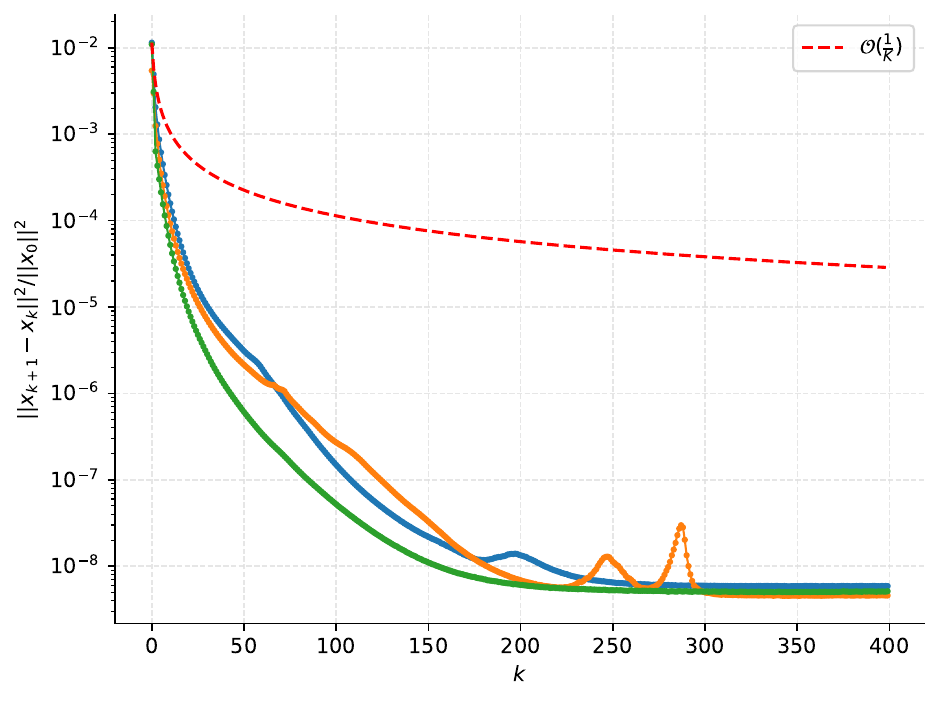}}
    \hskip2mm 
    \subfloat[\dystv]{\includegraphics[width=0.31\textwidth]{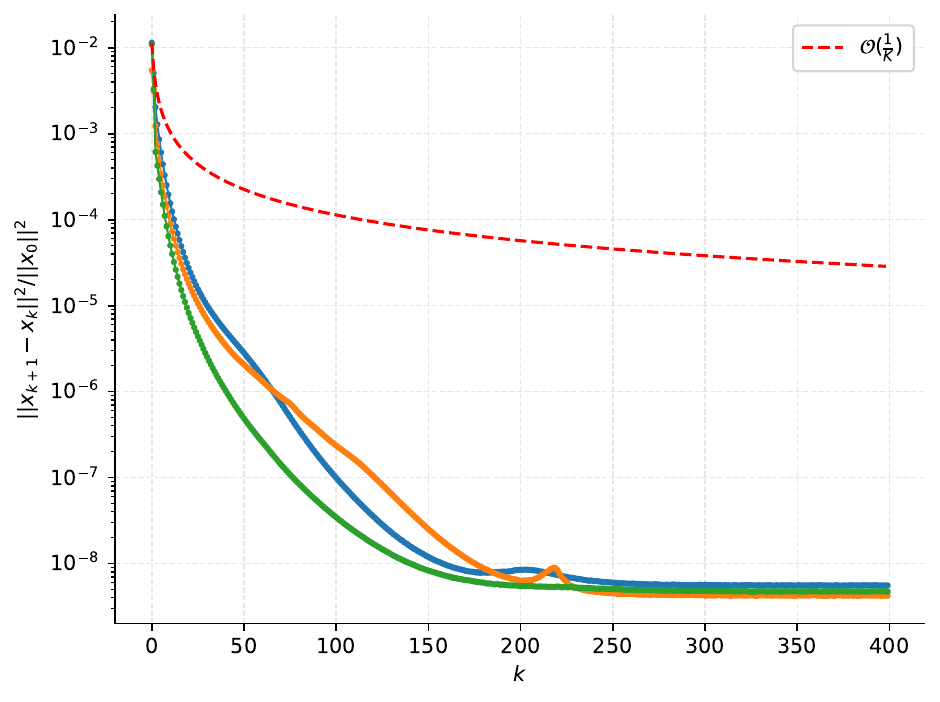}}
    \hskip2mm
    \subfloat[PnP-PGD]{\includegraphics[width=0.31\textwidth]{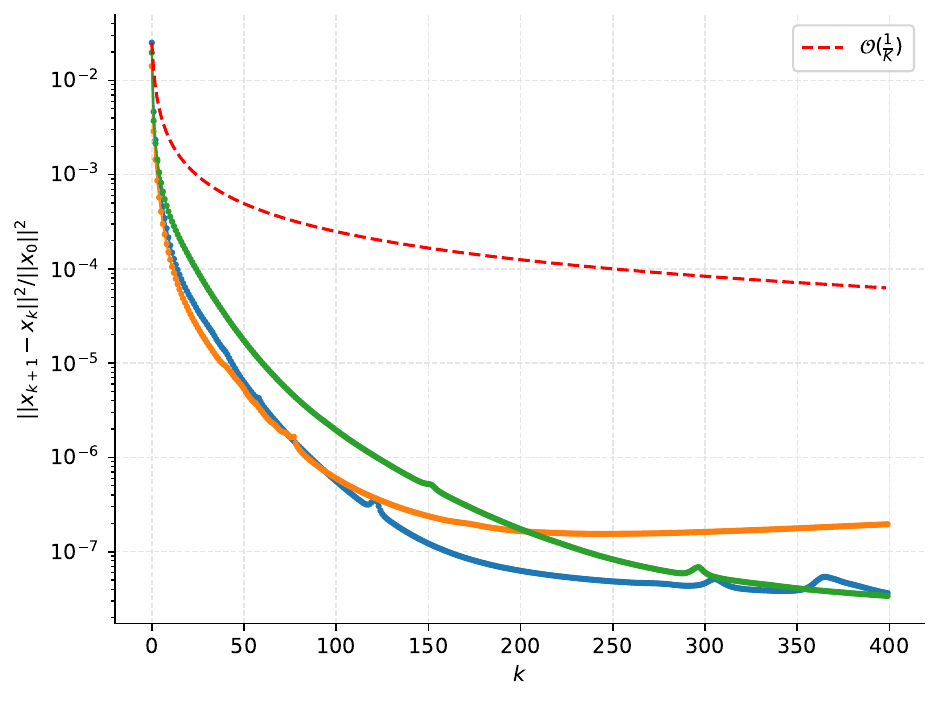}}
    \\
    \subfloat[GS-PnP]{\includegraphics[width=0.31\textwidth]{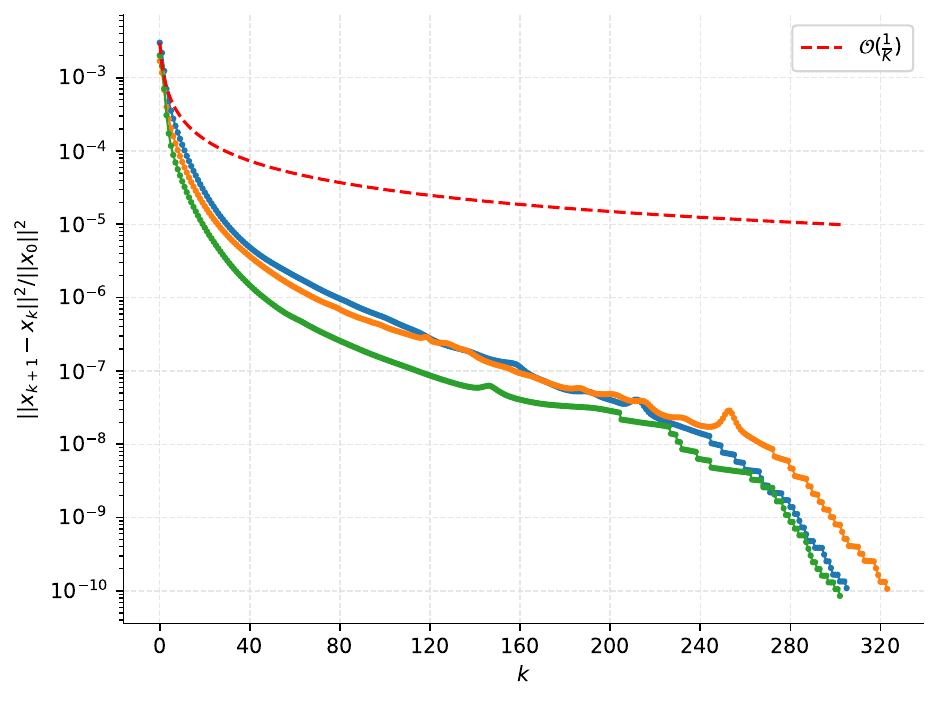}}
    \hskip2mm 
    \subfloat[DPIR]{\includegraphics[width=0.31\textwidth]{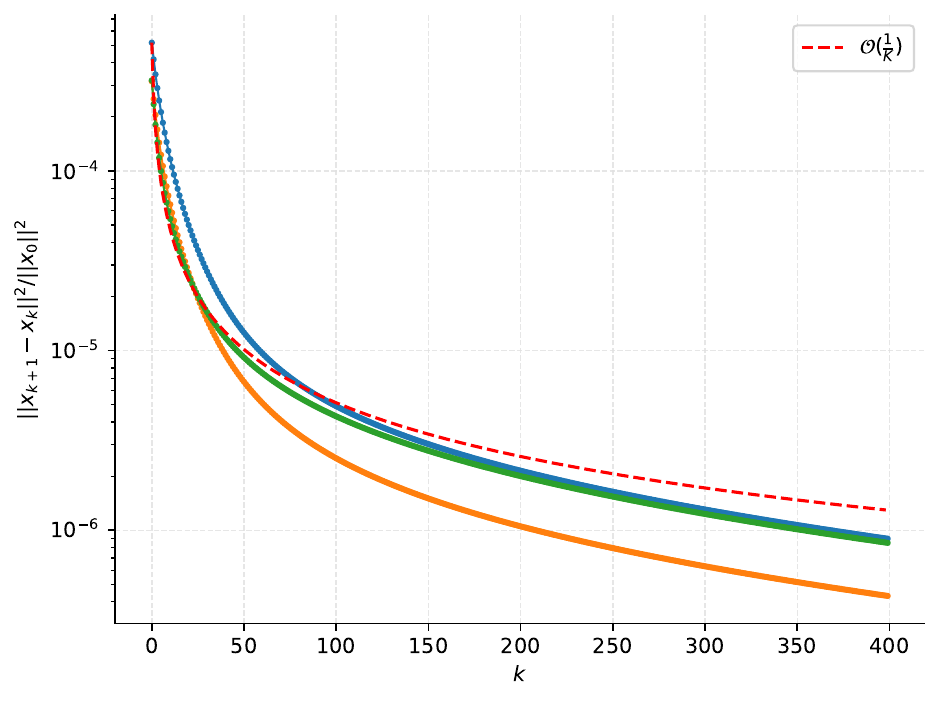}}
    \hskip2mm
    \subfloat[PnP-DRS]{\includegraphics[width=0.31\textwidth]{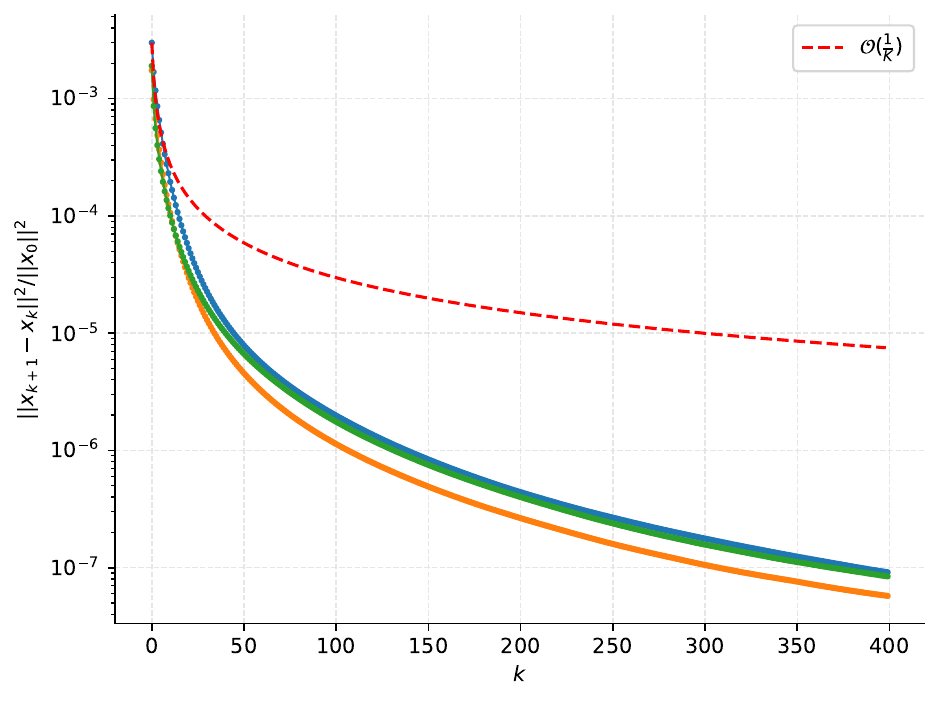}}
    \\
    \caption{Residuals $\norm{\xkp-\xk}^2/\norm{x_0}^2$ for Laplace--Gaussian deblurring with $\sigma_g=\sigma_l=7.65$. Each solid line corresponds to one image of set3c. All the methods converge in terms of residual. GS-PnP's backtracking procedure leads to small step sizes and a rapid residual decay near the end.}    
    \label{fig:comp_deblur_conv}
\end{figure}

\begin{figure}[!hbtp]
    \centering
    \subfloat[\dysgs]{\includegraphics[width=0.31\textwidth]{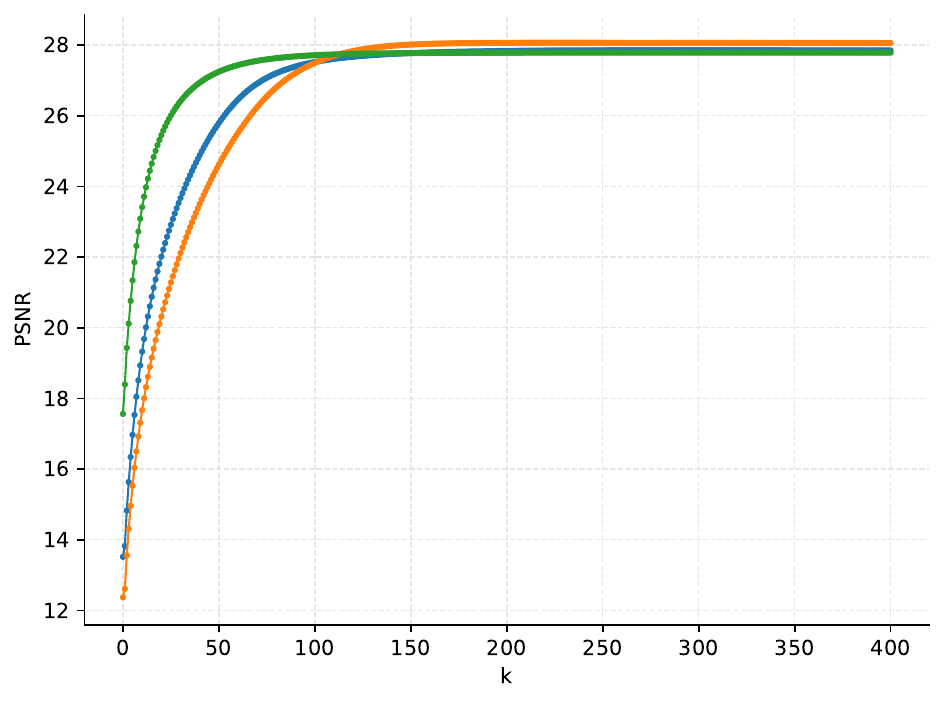}}
    \hskip2mm 
    \subfloat[\dystv]{\includegraphics[width=0.31\textwidth]{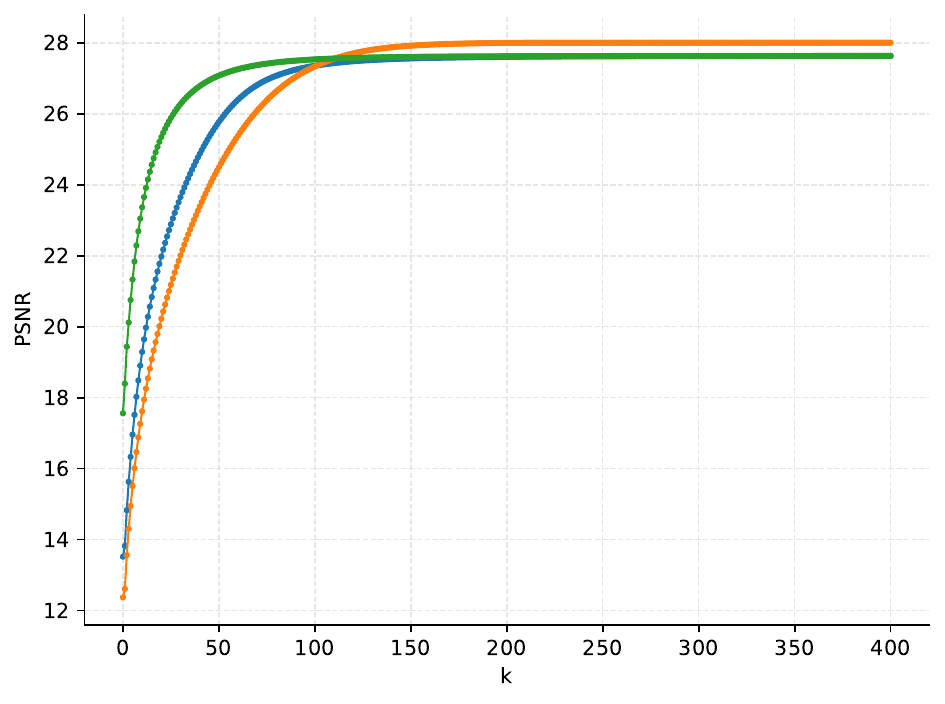}}
    \hskip2mm
    \subfloat[PnP-PGD]{\includegraphics[width=0.31\textwidth]{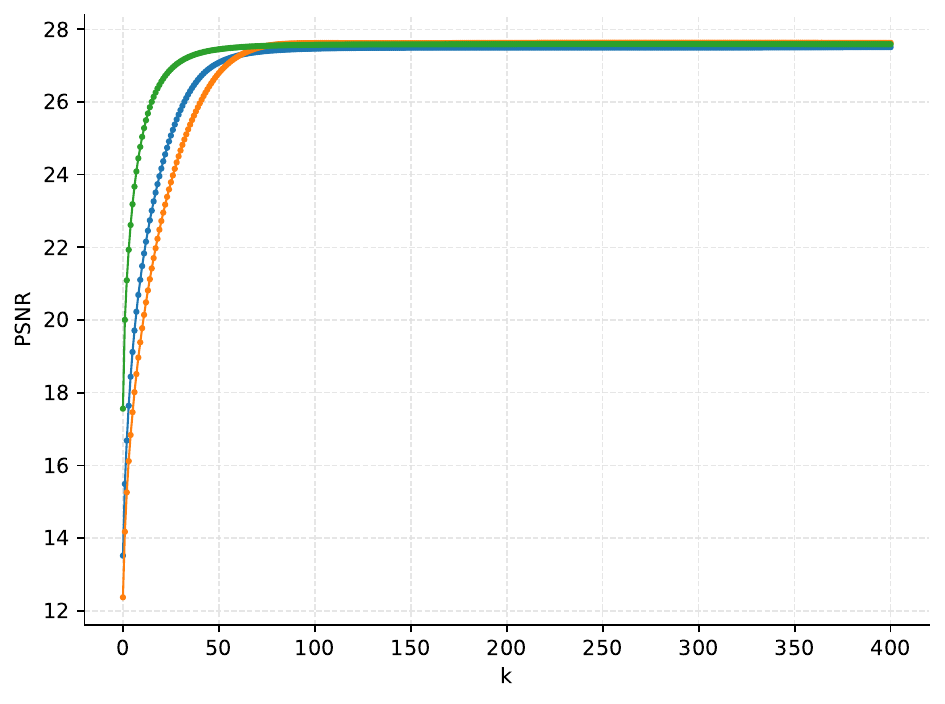}}
    \\
    \subfloat[GS-PnP]{\includegraphics[width=0.31\textwidth]{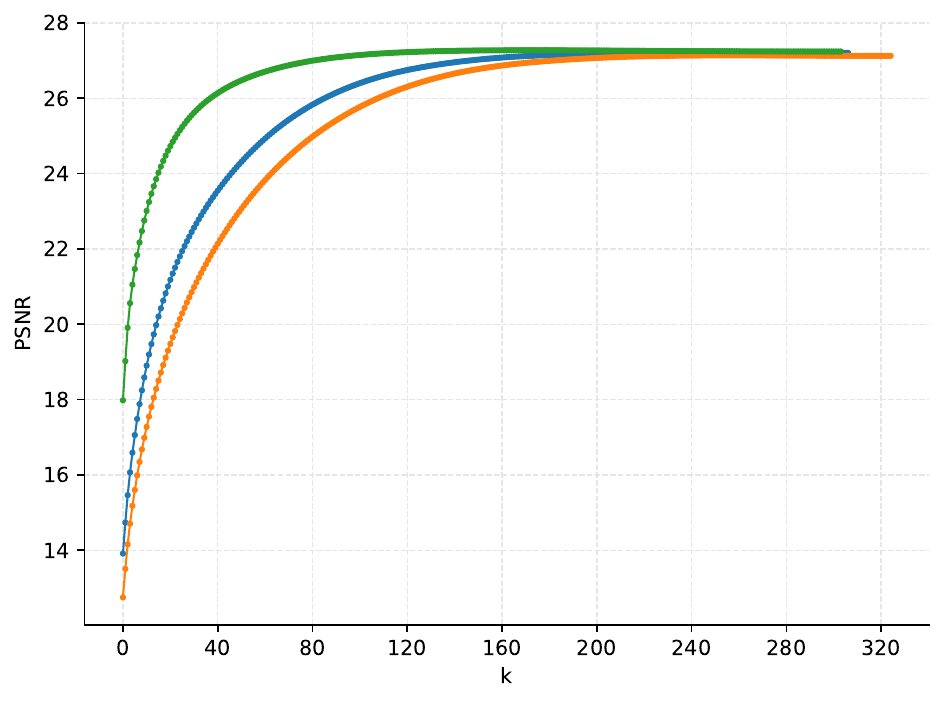}}
    \hskip2mm 
    \subfloat[DPIR]{\includegraphics[width=0.31\textwidth]{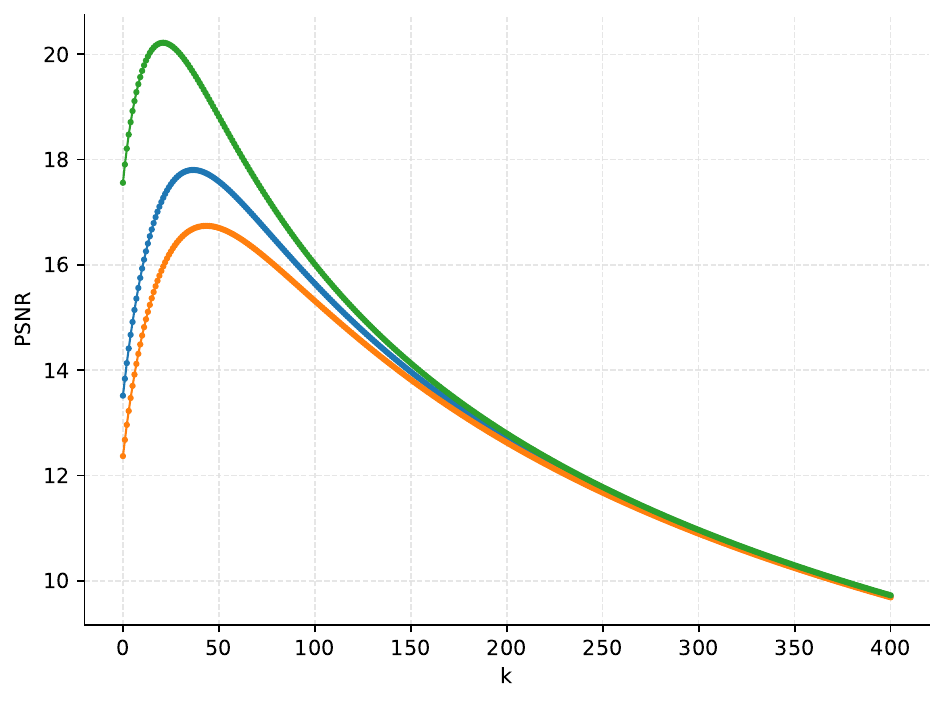}}
    \hskip2mm
    \subfloat[PnP-DRS]{\includegraphics[width=0.31\textwidth]{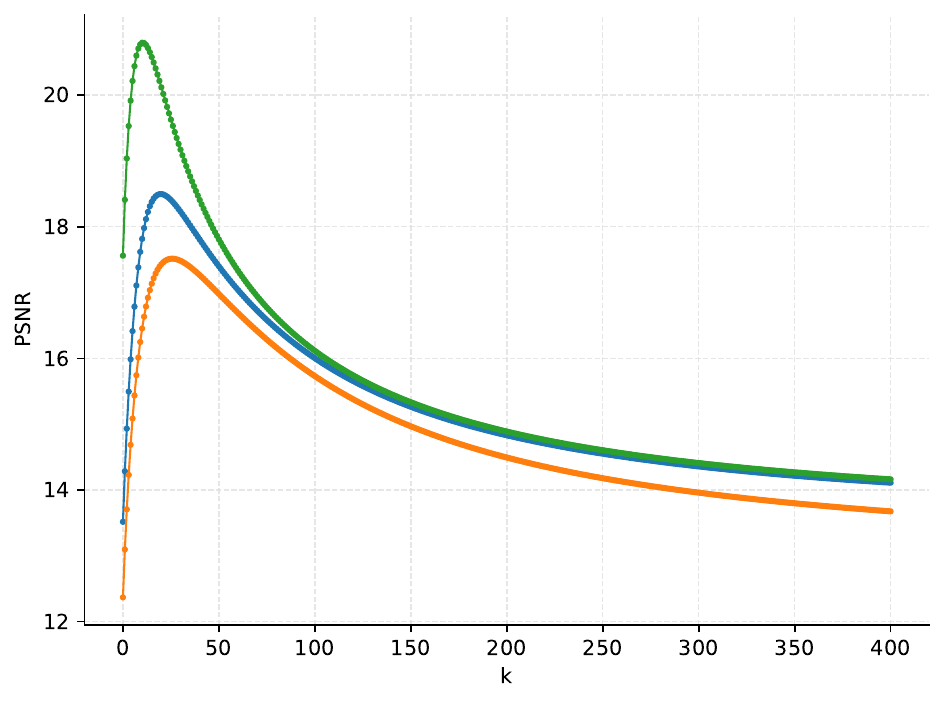}}
    \\
    \caption{PSNR curves for Laplace--Gaussian deblurring with $\sigma_g=\sigma_l=7.65$. Each solid line corresponds to one image of set3c. While the PnP methods using the infimal convolution converge to high PSNR, the Gaussian fidelity methods DPIR and PnP-DRS slowly diverge.}
    \label{fig:comp_deblur_psnr}
\end{figure}

\subsubsection{Super-resolution}\label{ssec:LGSR}
For the image super-resolution task, we again take the setup of \cite{hurault2021gradient}. The forward operator $A$ is given by first applying the first camera blur kernel for antialiasing, then downsampling with scale factors $s=2$ and $s=3$. The forward operator and its transpose can again be easily computed using Fourier transforms.

\begin{table}[!htbp]
        \centering
        \caption{Average PSNR (dB) across different methods for super-resolution under Laplace--Gaussian mixed noise on the CBSD10 dataset. The best results are highlighted in bold, and the second-best results are underlined. The infimal convolution PnP methods have better performance, particularly in the presence of larger Laplace noise.}
        \label{tab:sisr_laplace}
        \footnotesize
        \setlength{\tabcolsep}{6pt} 
        \begin{tabular}{ccccccccc}
            \toprule
             Scale & \multicolumn{4}{c}{$s = 2$} & \multicolumn{4}{c}{$s = 3$} \\
            \cmidrule(lr){2-5} \cmidrule(lr){6-9}
            $\sigma_g$ & 2.55 & 2.55 & 7.65 & 7.65 & 2.55 & 2.55 & 7.65 & 7.65 \\
            $\sigma_l$ & 2.55 & 7.65 & 2.55 & 7.65 & 2.55 & 7.65 & 2.55 & 7.65 \\
            \midrule
           \dysgs & \underline{30.18} & \textbf{28.85} & \textbf{29.37} & \underline{28.45} & \underline{27.14} & \underline{26.42} & \underline{26.79} & \underline{26.28} \\
            \dystv & \textbf{30.21} & 28.83 & \underline{29.34} & 28.35 & \textbf{27.20} & \textbf{26.49} & \textbf{26.83} & \textbf{26.32} \\
            PnP-PGD & 29.90 & \underline{28.84} & 29.33 & \textbf{28.49} & 26.99 & 26.33 & 26.68 & 26.12\\
            GS-PnP  & 29.74 & 28.13 & 28.85 & 28.06 & 26.93 & 25.98 & 26.51 & 25.77 \\
            DPIR & 28.98 & 28.35 & 28.61 & 28.11 & 25.98 & 25.67 & 25.83 & 25.52 \\
            PnP-DRS  & 24.94 & 22.37 & 23.41 & 21.54 & 23.48 & 21.97 & 22.66 & 21.39 \\
            % PnP-DRSdiff & 24.21 & 20.72 & 22.01 & 19.80 & 23.33 & 21.47 & 22.31 & 20.81 \\
            \bottomrule
        \end{tabular}
    \end{table}

\Cref{tab:sisr_laplace} gives the average PSNR over the CBSD10 dataset for the two scales. Similarly to deblurring, the methods using $\psi_g$ perform slightly worse than those using $\psi_{\mathrm{IC}}$, with the gap becoming larger as the scale $\sigma_l$ increases. PnP-DYS and PnP-PGD both have similar performance, favoring PnP-DYS for the larger scale factor $s=3$, where the inverse problem is more ill-conditioned and the additional prior helps stabilize the reconstruction.

\subsubsection{High noise}\label{ssec:LGHighNoise}
To evaluate the robustness of the infimal convolution framework under severe degradation, we consider deblurring with high-noise, taking $\sigma_g = \sigma_l = 25.5$. This corresponds to an effective $17\%$ noise. An extended hyperparameter search is done and detailed in \Cref{sec:params-lg}.

The average PSNR per kernel is summarized in Table \ref{tab:psnr_transposed}. PnP-DRS and DPIR both diverge and we do not report their results. The proposed PnP-DYS and PnP-PGD methods with the infimal convolution fidelity are both stable and converge to similar PSNRs. Moreover, the average reconstruction quality using the infimal convolution fidelity is approximately 1dB better than the Gaussian fidelity GS-PnP. 

While the PSNR figures are quite similar, PnP-DYS and PnP-PGD have different qualitative convergence. \Cref{fig:comp_deblur_large} plots the reconstructions of the tiki image from CBSD10 under high noise deblurring, and points out artifacts that occur for PnP-PGD but not PnP-DYS. PnP-PGD tends to oversmooth the image, leading to a loss of structural details and textures. GS-PnP suffers from significant graininess throughout the image, coming from the heavy tails of the Laplace noise not being captured by the Gaussian fidelity. In contrast, PnP-DYS effectively suppresses the heavy mixed noise while successfully restoring sharp edges and fine textures.

\begin{table}[htbp!]
    \centering
    \caption{Average PSNR (dB) results for image deblurring under heavy Laplace--Gaussian mixed noise on the CBSD10 dataset, with $\sigma_l=\sigma_g=25.5$. Results are reported across different blur kernels. The best results are highlighted in bold, and the second-best results are underlined. }
    \label{tab:psnr_transposed}
    \fontsize{6.5pt}{7.5pt}\selectfont % 
    \renewcommand{\arraystretch}{1.3} % 
    \setlength{\tabcolsep}{5.5pt} % 
    \begin{tabular}{lcccccccccc}
        \toprule
        Kernel & $1$ & $2$ & $3$ & $4$ & $5$ & $6$ & $7$ & $8$ & $9$ & $10$ \\
        \midrule
        \dysgs  & \textbf{24.57} & \textbf{24.48} & \underline{24.90} & \textbf{24.11} & \underline{25.41} & \textbf{25.24} & \textbf{24.99} & \textbf{24.57} & \textbf{24.07} & \underline{24.94}\\
        \dystv  & 24.46 & 24.37 & 24.85 & \underline{24.07} & 25.28 & 25.07 & \underline{24.90} & \underline{24.53} & 24.00 & 24.91\\
        PnP-PGD & \underline{24.56} & \underline{24.40} & \textbf{25.16} & 23.92 & \textbf{25.66} & \underline{25.18} & 24.77 & 24.35 & \underline{24.04} & \textbf{25.48}\\
        GS-PnP   & 23.36 & 23.52 & 24.34 & 23.19 & 24.44 & 24.00 & 23.83 & 23.62 & 23.85 & \underline{24.94} \\
        \bottomrule
    \end{tabular}
\end{table}

\begin{figure}[!hbtp]
    \centering
    % --- 
    \subfloat[Ground Truth]{    \tikzmarknode{gt}{\includegraphics[width=0.28\textwidth]{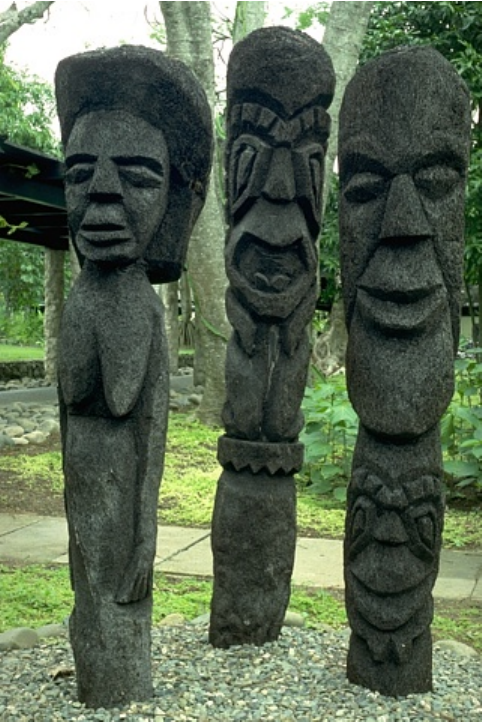}
        \label{fig:gt}}
    }
    \hskip2mm
    \subfloat[Corrupted(13.54dB)]{
        \includegraphics[width=0.28\textwidth]{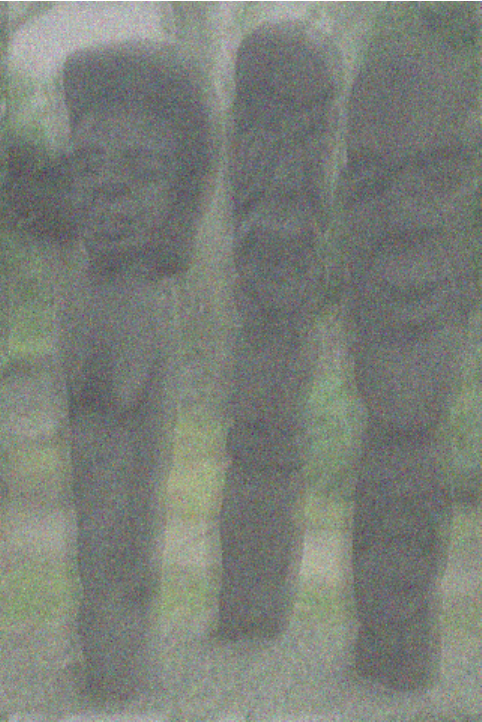}
        \label{fig:corrupted}
    }
    \hskip2mm
    \subfloat[\dysgs(21.55dB)]{
        \includegraphics[width=0.28\textwidth]{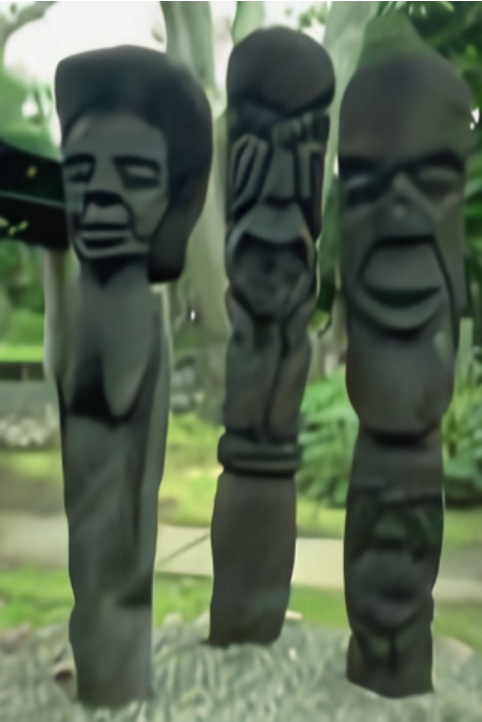}
        \label{fig:method_b}
    }
    \\
    \subfloat[\dystv(21.57dB)]{
        \includegraphics[width=0.28\textwidth]{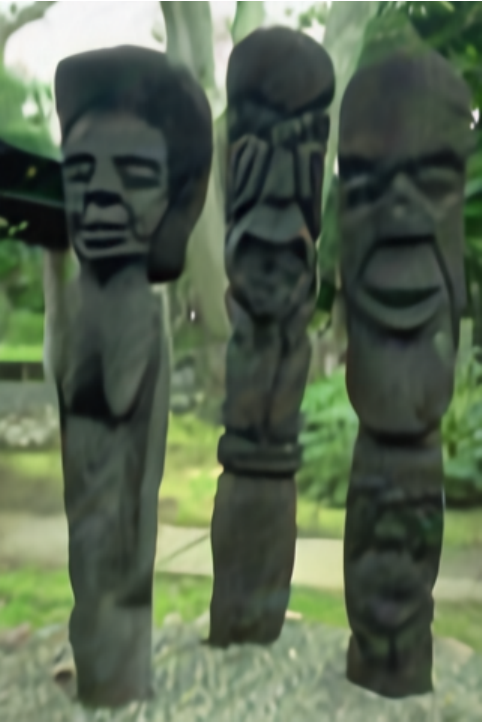}
        \label{fig:method_a}
    }
    \hskip2mm
    \subfloat[PnP-PGD(21.35dB)]{
        \includegraphics[width=0.28\textwidth]{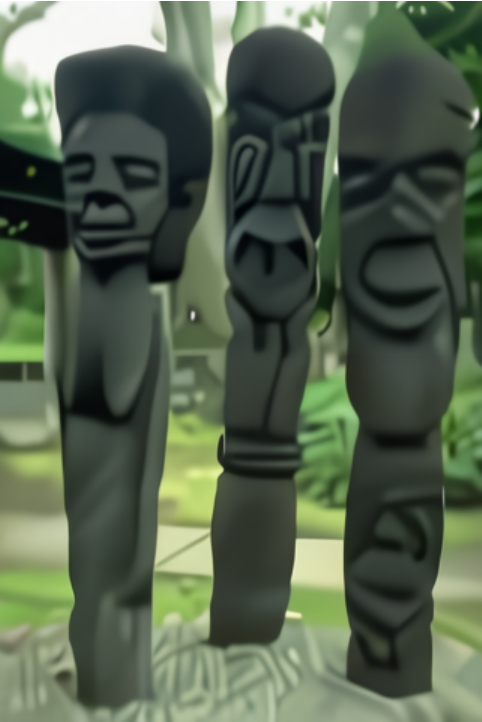}
        \label{fig:conv}
    }
    \hskip2mm
    \subfloat[GS-PnP(21.22dB)]{
        \includegraphics[width=0.28\textwidth]{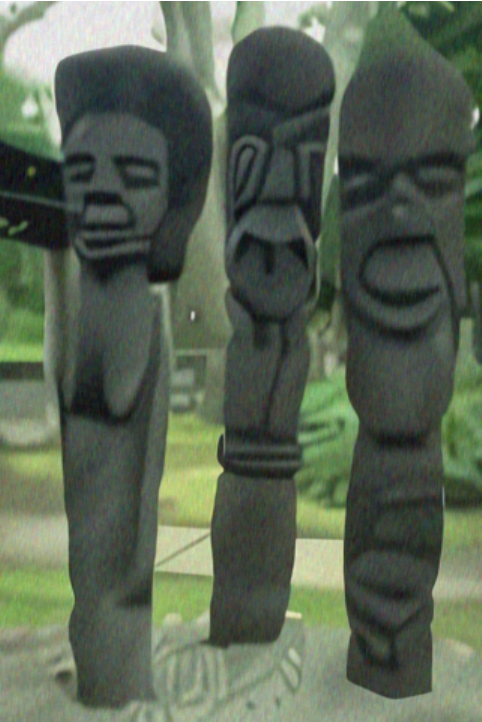}
        \label{fig:psnr}
    }
    \begin{tikzpicture}[remember picture,overlay]
    \node at ($(gt.center)+(40pt,-30pt)$) (c) {};
    \draw[->,red,very thick] (c)++(10pt,30pt) to (c);
    \node at ($(gt.center)+(-24pt,-30pt)$) (b) {};
    \draw[->,red,very thick] (b)++(-20pt,0pt) to (b);
    \node at ($(gt.center)+(24pt,74pt)$) (a) {};
    \draw[->,red,very thick] (a)++(-20pt,0pt) to (a);
    \node at ($(gt.center)+(-50pt,-10pt)$) (d) {};
    \draw[->,red,very thick] (d)++(0pt,20pt) to (d);
    \end{tikzpicture}

    \caption{Visual comparison of image deblurring under heavy Laplace--Gaussian noise, with red arrows on the ground truth to indicate regions of significant textural difference. While PnP-DYS and PnP-PGD use the same fidelity, PnP-PGD tends to oversmooth in the tiki bodies. A common misreconstructed region is above the rightmost tiki head; \dystv has the most green background in this region, while PnP-PGD completely smooths this region away. The mismatched fidelity of GS-PnP manifests as overall graininess, which only slightly affects the PSNR value.}
    \label{fig:comp_deblur_large}
\end{figure}
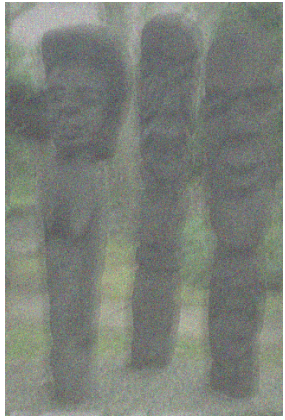
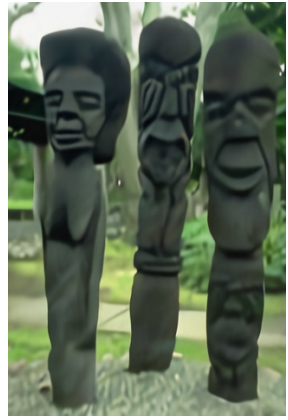
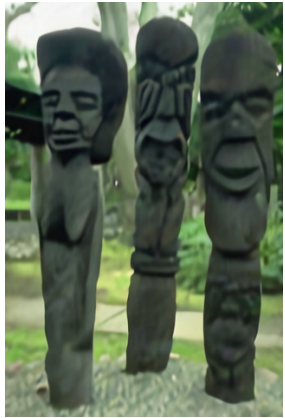
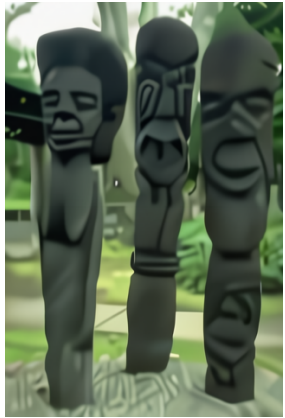
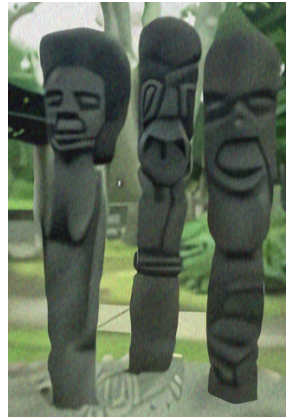

To further evaluate the robustness and stability of our framework, \Cref{fig:comp_deblur_dys} illustrates a comparison between \dystv and \dysgs under an extreme noise regime with $\sigma_g=7.65$ and $\sigma_l = 68.85$, corresponding to 38\% noise. This requires an aggressive fidelity parameter setting with $\lambda_g = 10.0$ and $\lambda_l=7.65$. While \dysgs achieves a slightly higher PSNR value in the early stages of iteration, its convergence behavior exhibits fluctuations as the iteration step increases. This instability is also visually reflected in subfigure (d), where the reconstructed image suffers from a brightness drift compared to the reference image.
In contrast, by replacing one of the deep denoisers with a TV denoiser, \dystv demonstrates stability under these conditions. This comparison reveals that although deep denoisers provide powerful representative capabilities, the inclusion of a more stable TV prior in the PnP-DYS framework can offer a more reliable and robust choice where noise levels are extremely large.

\begin{figure}[!hbtp]
    \centering
    \subfloat[Ground Truth]{\includegraphics[width=0.23\textwidth]{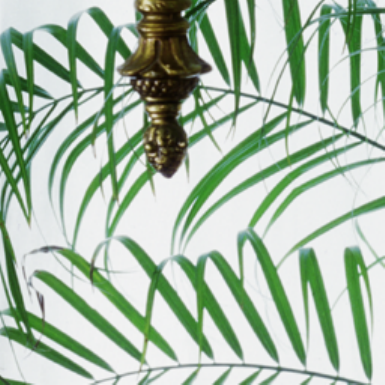}}
    \hskip2mm
    \subfloat[Corrupted (7.72dB)]{\includegraphics[width=0.23\textwidth]{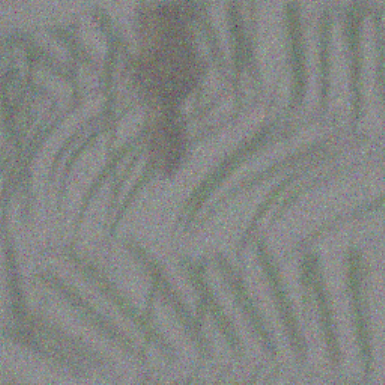}}
    \hskip2mm 
    \subfloat[\dystv (18.59dB)]{\includegraphics[width=0.23\textwidth]{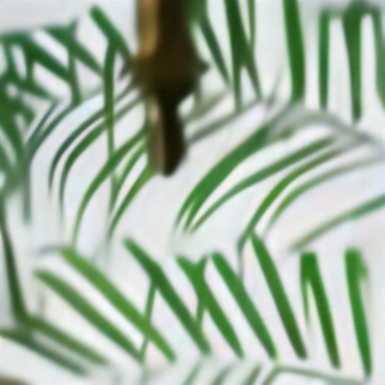}}
    \hskip2mm
    \subfloat[\dysgs (17.17dB)]{\includegraphics[width=0.23\textwidth]{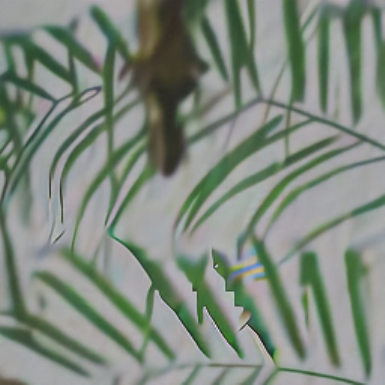}}
    \\
    \subfloat[Convergence of \dystv]{\includegraphics[width=0.23\textwidth]{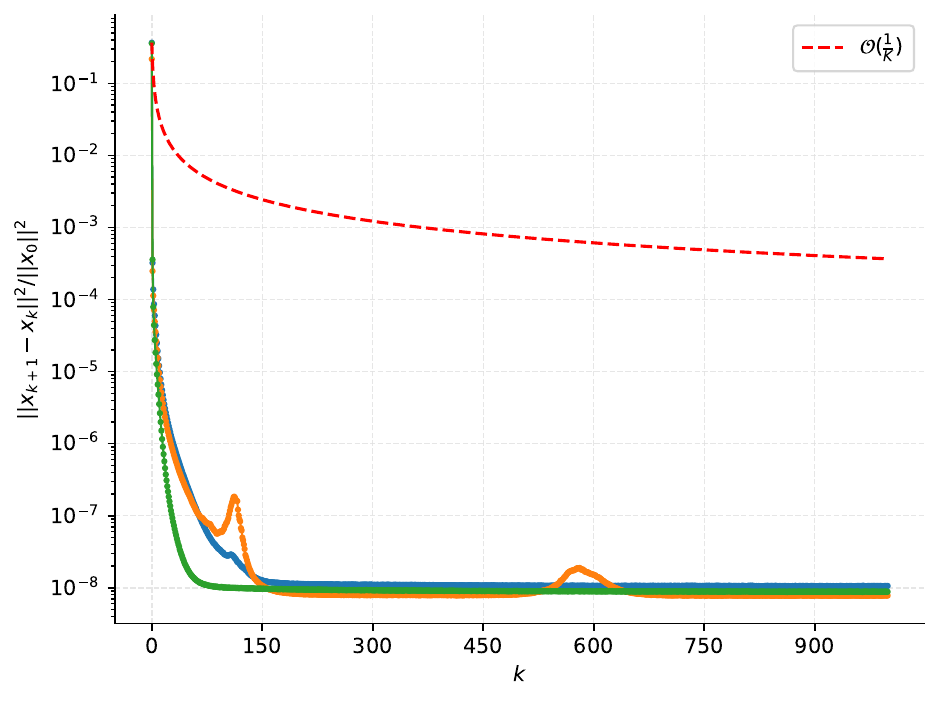}}
    \hskip2mm
    \subfloat[Convergence of \dysgs]{\includegraphics[width=0.23\textwidth]{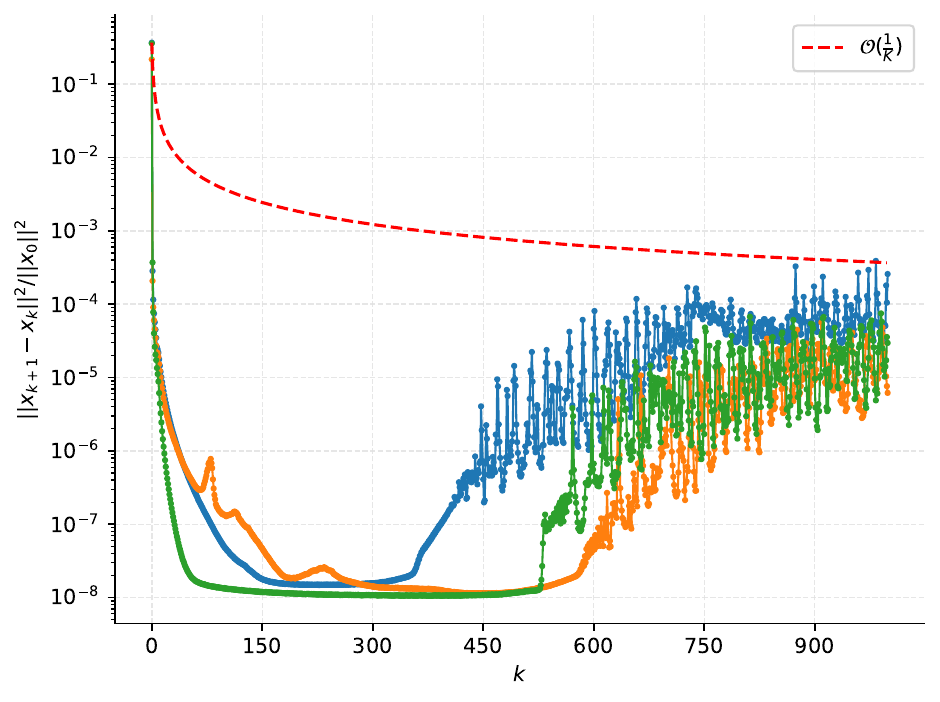}}
    \hskip2mm 
    \subfloat[PSNR of \dystv]{\includegraphics[width=0.23\textwidth]{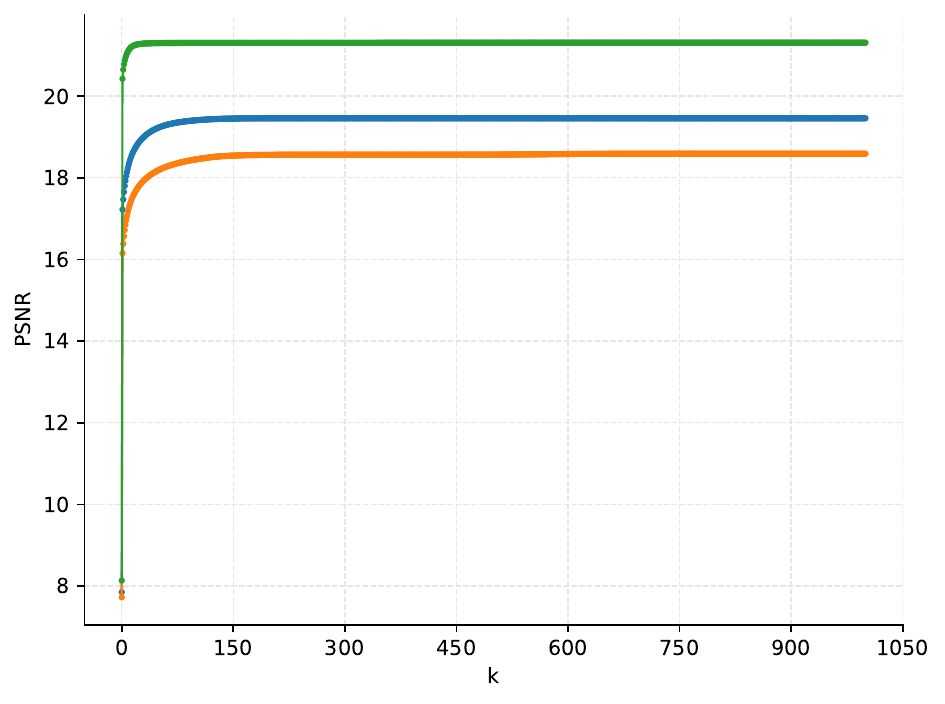}}
    \hskip2mm
    \subfloat[PSNR of \dysgs]{\includegraphics[width=0.23\textwidth]{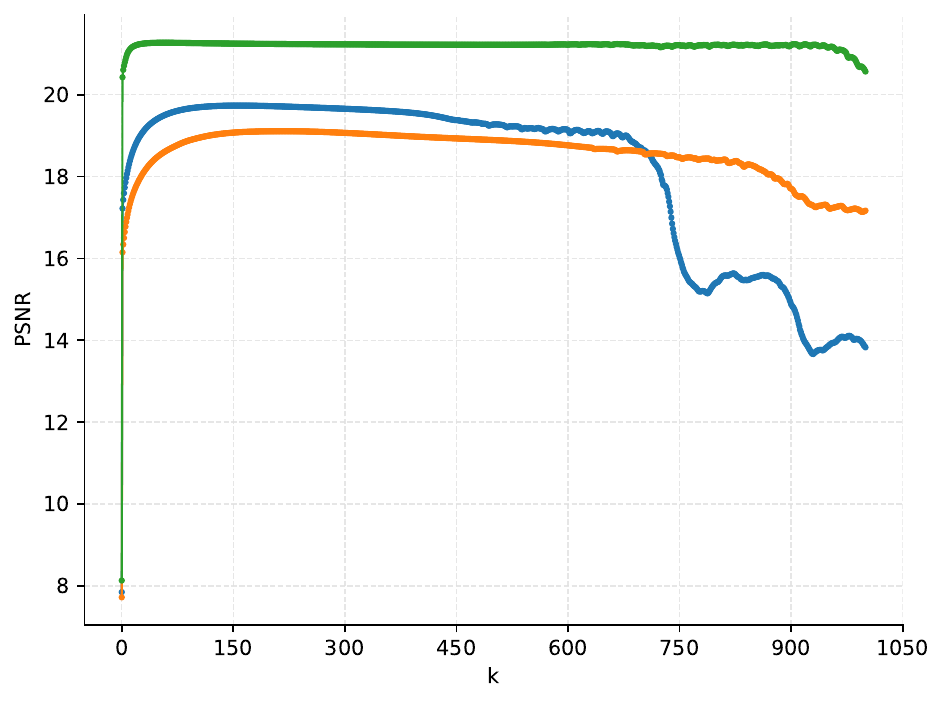}}
    \\
    \caption{Comparison of image deblurring results corrupted by large noise on the leaves image of PnP-DYS methods, evaluated up to 1000 iterations. The noise scales are $\sigma_g=7.65$ and $\sigma_l = 68.85$, corresponding to $38\%$ standard deviation. In this very high noise regime, \dystv is more stable compared to \dysgs, where the latter begins to have divergent artifacts.}
    \label{fig:comp_deblur_dys}
\end{figure}

\subsection{Poisson--Gaussian mixed noise}\label{ssec:PoissonGaussian}
We now evaluate various PnP methods under Poisson--Gaussian mixed noise, for deblurring in \Cref{ssec:PGDeblur} and super-resolution in \Cref{ssec:PGSR}. We consider PnP-DYS, PnP-PGD, and GS-PnP, dropping the PnP-DRS and DPIR methods due to poor convergence. The observed image $y$ is modeled as:
\beq
y = \frac{1}{\sigma_p} z + \omega,\quad z \sim \mathrm{Pois}(\sigma_p Ax), \quad \omega \sim \gN(0, \sigma_g^2 I),
\eeq
where $A$ again represents the linear forward operator. The parameter $\sigma_p$ acts as the peak signal intensity, where smaller $\sigma_p$ corresponds to a lower photon count and thus more severe Poisson noise. The noise parameters are varied over $\sigma_g \in \{2.55, 7.65, 12.75\}$ and $\sigma_p \in \{20, 40, 60, 100, 200\}$. While the gradient of the Laplace--Gaussian infimal convolution term could be computed in closed form, the gradient of the Poisson--Gaussian term requires a short inner loop, which can be computed using \labelcref{eqs:PoissonGaussianGrad}. In practice, a Newton iteration inner loop for evaluating the gradient of $\psi_{\mathrm{IC}}$ requires at most 10 iterations to converge up to a tolerance of $10^{-6}$.

\subsubsection{Deblurring}\label{ssec:PGDeblur}

We consider the setup as in \Cref{ssec:LGDeblur} with a camera blur kernel and varying noise levels. \Cref{tab:performance_deblur_poisson} presents the average PSNR values over the set3c dataset for PnP-DYS and PnP-PGD using the infimal convolution fidelity, and GS-PnP using the Gaussian fidelity. The infimal convolution fidelity consistently outperforms the Gaussian fidelity on this problem, despite PSNR measuring the average $\ell_2$ error, suggesting that proper modeling is necessary for good reconstructions. We additionally find that PnP-DRS and DPIR diverge for the Gaussian fidelity with this noise. 

For this experiment, the infimal convolution PnP methods all have fairly similar performance. In the high Poisson noise regime, corresponding to small $\sigma_p$, PnP-PGD performs slightly better in terms of PSNR. PnP-DYS is best in the medium noise regime, and also when the noise is mainly Poissonian, indicated by the superior performance for $\sigma_g = 2.55$. The stability of each of these methods across both Laplace--Gaussian and Poisson--Gaussian suggests that while deep Gaussian denoisers are not trained on these different noises, the image priors they induce are still sufficient to provide reasonable reconstructions.

\begin{table}[h!]
    \centering
    \caption{Average PSNR over set3c images for deblurring across different noise levels for kernel $1$, with various Poisson--Gaussian noise. Higher $\sigma_p$ indicates less Poisson noise. PnP-DYS and PnP-PGD use the infimal convolution fidelity, while GS-PnP uses a Gaussian fidelity. Quantitative results of infimal convolution-based PnP are all similar, and beat the Gaussian fidelity used in GS-PnP.}
    \label{tab:performance_deblur_poisson}
    \small
    \begin{tabular}{cc cccc} 
        \toprule
        \multirow{2}{*}{$\sigma_g$} &\multirow{2}{*}{$\sigma_p$} & \multicolumn{4}{c}{Methods} \\
        \cmidrule(l){3-6}
        & & \dysgs & \dystv & PnP-PGD & GS-PnP \\
        \midrule 
        \multirow{5}{*}{2.55}  & 20 & 22.49 & \underline{22.50} & \textbf{22.62} & 21.46 \\
        & 40 & \textbf{24.33} & 24.24 & \underline{24.25} & 23.31 \\
        & 60 & \textbf{25.43}  & \underline{25.32} & 25.20 & 24.19 \\
        & 100 & \textbf{26.48} & \underline{26.44} & 26.09 & 25.16\\
        & 200 & \underline{27.99} & \textbf{28.00} & 27.30 & 26.46\\
        \midrule \midrule
        \multirow{5}{*}{7.65}  & 20 & \underline{22.44} & \underline{22.44} & \textbf{22.57} & 21.39\\
        & 40 & \textbf{24.16} & 24.12 & \underline{24.13} & 23.21\\
        & 60 & \textbf{25.23} & \underline{25.13} & 25.00 & 24.08\\
        & 100 & \textbf{26.13} & \underline{26.08} & 25.88 & 24.99\\
        & 200 & \underline{27.30} & \textbf{27.31} & 26.81 & 26.15\\
        \midrule \midrule
        \multirow{5}{*}{12.75}  & 20 & 22.30 & \underline{22.32} & \textbf{22.49} & 21.30\\
        & 40 & \textbf{24.16} & 23.84 & \underline{23.91} & 23.04 \\
        & 60 & \textbf{24.80} & \underline{24.73} & 24.70 & 23.84 \\
        & 100 & \textbf{25.48} & \underline{25.44} & 25.37 & 24.65 \\
        & 200 & \underline{25.94} & \textbf{26.02} & \textbf{26.02} & 25.59\\
        \bottomrule
    \end{tabular}
\end{table}

\subsubsection{Super-resolution}\label{ssec:PGSR}

For super-resolution, we use the same setup as \Cref{ssec:LGSR} with the Poisson--Gaussian noise, and evaluate on CBSD10 with the additional blur kernel. \Cref{tab:performance_sisr_poisson} summarizes the PSNR and runtime of the PnP methods in the previous section, along with the runtime per 100 iterations. As in other sections, PnP-DYS and PnP-PGD with the infimal convolution fidelity outperform GS-PnP with the Gaussian fidelity, with the gap widening for larger Poisson noise.

Computational overhead of the infimal convolution and PnP-DYS can be seen separately from the average runtime. The Newton iteration inner loop manifests as the difference between PnP-PGD and GS-PnP: we observe that it incurs at most a 2\% wall-clock time increase. Increasing the number of priors has a more significant impact on computational cost. \dysgs requires two denoiser evaluations per iteration, leading to nearly double the cost of PnP-PGD and GS-PnP, while computing the proximal of the TV term in \dystv requires a smaller 30\% increase in time.

\begin{table}[htbp!]
    \centering
    \caption{PSNR and average runtime for super-resolution under Poisson--Gaussian noise, averaged over CBSD10 dataset and with additional blur kernel $1$ for different scales $s$. We observe again that the infimal convolution methods are able to outperform the Gaussian fidelity based GS-PnP method once leaving the low Gaussian plus low Poisson noise regime. Moreover, the infimal convolution methods are all competitive with each other. \dysgs takes around twice as long per iteration as PnP-PGD due to having to evaluate two denoisers.}
    \label{tab:performance_sisr_poisson}
    \footnotesize
    \setlength{\tabcolsep}{0pt} 
    \begin{tabular*}{\textwidth}{@{\extracolsep{\fill}}cc cccc cccc} 
        \toprule
        \multirow{2}{*}{$\sigma_g$} & \multirow{2}{*}{$\sigma_p$} & \multicolumn{4}{c}{$s=2$} & \multicolumn{4}{c}{$s=3$} \\
        \cmidrule(lr){3-6} \cmidrule(lr){7-10}
        & & DYS\textsuperscript{GS-GS} & DYS\textsuperscript{GS-TV} & PGD & GS & DYS\textsuperscript{GS-GS} & DYS\textsuperscript{GS-TV} & PGD & GS \\
        \midrule
        \multirow{5}{*}{2.55} & 20 & \textbf{24.38} & 24.19 & \underline{24.31} & 23.77 & \underline{23.26} & \textbf{23.27} & 23.07 & 22.74 \\
        & 40 & \textbf{25.32} & 25.23 & \underline{25.31} & 24.82 & \textbf{24.21} & \underline{24.20} & \underline{24.20} & 23.73 \\
        & 60 & \textbf{25.87} & 25.81 & \underline{25.85} & 25.39 & \textbf{24.78} & \underline{24.74} & \underline{24.74} & 24.31 \\
        & 100 & \textbf{26.42} & \underline{26.39} & 26.37 & 25.95 & \textbf{25.44} & \underline{25.42} & 25.31 & 24.87 \\
        & 200 & \underline{27.11} & \textbf{27.15} & 26.94 & 26.90 & \textbf{26.11} & \underline{26.10} & 25.96 & 25.86 \\
        \midrule
        \multirow{5}{*}{7.65} & 20 & \textbf{24.32} & 24.12 & \underline{24.25} & 23.71 & \underline{23.20} & \textbf{23.22} & 23.05 & 22.69 \\
        & 40 & \textbf{25.25} & 25.13 & \underline{25.23} & 24.74 & \textbf{24.14} & \underline{24.13} & \textbf{24.14} & 23.67 \\
        & 60 & \textbf{25.76} & \underline{25.68} & \textbf{25.76} & 25.32 & \textbf{24.64} & \underline{24.63} & \textbf{24.64} & 24.19 \\
        & 100 & \textbf{26.23} & \underline{26.16} & \textbf{26.23} & 25.83 & \textbf{25.29} & \underline{25.27} & 25.17 & 24.74 \\
        & 200 & \textbf{26.82} & \textbf{26.82} & \underline{26.73} & 26.59 & \textbf{25.87} & \underline{25.86} & 25.75 & 25.45 \\

        \midrule
        \multirow{5}{*}{12.75} & 20 & \textbf{24.21} & 23.97 & \underline{24.14} & 23.62 & \underline{23.11} & \textbf{23.13} & 22.92 & 22.61 \\
        & 40 & \textbf{25.09} & 24.92 & 25.08 & 24.60 & 23.98 & \textbf{24.00} & \underline{23.99} & 23.53 \\
        & 60 & \underline{25.55} & 25.41 & \textbf{25.56} & 25.14 & \textbf{24.44} & \textbf{24.44} & \underline{24.43} & 24.02 \\
        & 100 & 25.76 & \underline{25.67} & \textbf{25.95} & 25.59 & \textbf{25.01} & \underline{24.98} & 24.91 & 24.49 \\
        & 200 & 26.07 & \underline{26.10} & \textbf{26.35} & 26.02 & \textbf{25.44} & \underline{25.42} & 25.33 & 24.71 \\

        \midrule
        \multicolumn{2}{c}{Time/100 iter (s)} & 8.27 & 5.92 & 4.54 & 4.46 & 8.11 & 5.75 & 4.55 & 4.50\\

        \bottomrule
    \end{tabular*}
\end{table}

\section{Conclusion}
This work extends the Plug-and-Play framework for imaging inverse problems to mixed noise corruption. By replacing the standard $\ell_2^2$ Gaussian fidelity with an infimal convolution, the variational model can be interpreted as a joint MAP estimation over the noises and prior. In the case of two log-concave noise distributions, we provide methods for computing the gradient of the fidelity, encompassing the Laplace--Gaussian and Poisson--Gaussian models as special cases, with the latter computed using a cheap inner Newton iteration. We further extend this to multiple priors using the Davis--Yin three-operator splitting, and show convergence under weaker assumptions compared to \cite{wu2024extrapolated}. Numerical experiments demonstrate that matching the infimal convolution fidelity provides significant quantitative performance increases over standard Gaussian fidelities. At noise levels above the usual level in PnP literature, they also increase robustness and stability without needing additional safeguards such as Armijo step sizes. Furthermore, we show that adding multiple priors leads to different stationary points with distinct qualitative properties, including fewer hallucination artifacts, reduced oversmoothing at high noise levels, and improved stability at high noise when one prior is chosen to be total variation.

This work primarily focuses on PnP methods utilizing gradients of the fidelity functions, limiting the choice of PnP algorithm. While different PnP methods with the same denoiser and parameter typically converge to the same stationary point, proximals of the infimal convolution fidelity would allow for use in methods such as PnP-DRS, which may have different convergence behavior. Further numerical analysis could include finding alternative constraints on the denoisers, such as weakening the Lipschitz constraint $L_g<1/2$ by imposing additional regularity on the second denoiser. Possible applications could include more complicated data modalities such as CT reconstruction, or replacing the denoiser priors with diffusion models such as in \cite{kawar2022denoising}. Other algorithmic approaches could consider alternative multi-operator splitting methods such as \cite{raguet2013generalized,ryu2020finding}, or splittings with Bregman modifications \cite{jian2026partially}.

\backmatter

\begin{appendices}

\section{Proofs}\label{sec:proofs}

\subsection{Proof of Proposition~\ref{prop:descent}} \label{appsec:ProofDescent}
We restate \Cref{prop:descent} here for convenience.

\begin{proposition*}
Consider applying the relaxed Davis--Yin splitting to the nonconvex function \labelcref{eq:proxyF} with step size $\gamma>0$ and relaxation $\eta \in (0,1]$, i.e. the iterations
\begin{equation}\label{eq:pnp-dys-prox}
\begin{cases}
    x^{k+1} = \prox_{\phi_1}(w^k), \\
    u^{k+1} = \prox_{\phi_2}(2x^{k+1}-w^k-\gamma\nabla \psi(x^{k+1})), \\
    w^{k+1} = w^k + \eta (u^{k+1}-x^{k+1}).
\end{cases}
\end{equation}
Assume the following conditions: 
\begin{enumerate}
    \item $\nabla\psi$ is $\beta$-Lipschitz,
    \item $\phi_1$ is $l_{\phi}$-weakly convex with $l_{\phi} \in [0,1)$, and moreover has $L_{\phi}$-Lipschitz gradient,
    \item $\phi_2$ is proper and closed with $\prox_{\phi_2}$ being nonempty.
\end{enumerate} 

Define the PnP-DYS energy function associated with
\eqref{eq:proxyF} by
\beq\label{eq:energy_func}
\begin{aligned}
\Theta_{\gamma,\eta}(x,u,w)
\coloneqq &\;
\psi(x) + \frac{1}{\gamma}\phi_1(x)
+
\frac{1}{\gamma}\phi_2(u) +
\frac{1}{2\gamma}
\norm{2x-u-w-\gamma\nabla\psi(x)}^2\\
&-
\frac{1}{2\gamma}
\norm{w-x+\gamma\nabla\psi(x)}^2
-
\frac{\eta}{\gamma}\norm{x-u}^2 .
\end{aligned}
\eeq
Further define the following descent constant
\begin{equation}
    \Lambda(\gamma, \eta) \coloneqq \frac{1}{\gamma} \left[-\frac{1 + l_\phi}{2} + \frac{1 - L_\phi^2}{\eta}\right] - \beta \left[1 + \frac{(1 - \eta + L_\phi)^2}{2\eta^2}\right].
\end{equation}
If $\Lambda(\gamma, \eta)>0$, then for all $k\geq 1$, the descent on $\Theta_{\gamma,\eta}$ holds
\begin{equation*}
\Theta_{\gamma,\eta}(x^{k+1},u^{k+1},w^{k+1})
-
\Theta_{\gamma,\eta}(x^k,u^k,w^k)
\leq
-\Lambda(\gamma,\eta)\norm{x^{k+1}-x^k}^2 .
\end{equation*}
\end{proposition*}
\begin{proof}
    Let us first define some intermediate variables
    \begin{equation}\label{eq:sk}
        \delta \coloneqq x^{k+1} - x^k,\quad d \coloneqq \nabla \phi_1(x^{k+1}) - \nabla \phi_1(x^{k}),\quad s\coloneqq w^{k} - w^{k-1}.
    \end{equation}
    By the optimality condition of $x^{k+1} = \prox_{\phi_1}(w^{k})$ (and similarly for $x^k$), and the update for $w^{k+1}$, we have that 
    \begin{equation*}
        \begin{cases}
            w^k = x^{k+1} + \nabla \phi_1(x^{k+1})\\
            w^{k-1} = x^{k} + \nabla \phi_1(x^{k})
        \end{cases} \quad \Longrightarrow \quad s = \delta + d = \eta (u^k - x^k).
    \end{equation*}
    We now show a decay on $\Theta_{\gamma,\eta}$. Rearranging $\Theta_{\gamma,\eta}$ gives the equivalent form
    \begin{equation}
        \begin{aligned}
            \Theta_{\gamma,\eta}(x,u,w)
            =&\;
            \frac{1}{\gamma}\pso(x)
            +
            \frac{1}{\gamma}\pst(u)
            +
            \psi(x)\\
            &+
            \frac{1-2\eta}{2\gamma}\norm{x-u}^2
            -
            \frac{1}{\gamma}\iprod{x-u}{w-x}
            +
            \iprod{\nabla\psi(x)}{u-x}.
        \end{aligned}
    \end{equation}
    From \cite[Lem. 3.3, Eq. A.4]{bian2021three}, the weak convexity assumption on $\phi_1$ and the existence of a proximal function gives the following upper bound on the Lyapunov decay,   
\[
\begin{aligned}
&\Theta_{\gamma,\eta}(x^{k+1},u^{k+1},w^{k+1})
-
\Theta_{\gamma,\eta}(x^k,u^k,w^k)\\
\leq\;&
-\frac{1-l_\phi}{2\gamma}
\norm{x^{k+1}-x^k}^2
+
\psi(x^{k+1})-\psi(x^k)
+
\left\langle
\nabla\psi(x^{k+1}),u^k-x^{k+1}
\right\rangle\\
&-
\left\langle
\nabla\psi(x^k),u^k-x^k
\right\rangle
+
\frac1\gamma\norm{x^{k+1}-u^k}^2
-
\frac{1-\eta}{\gamma}\norm{x^k-u^k}^2 .
\end{aligned}
\]

Since $\nabla\psi$ is
$\beta$-Lipschitz continuous, we have
\[
\psi(x^{k+1})-\psi(x^k)
-
\left\langle
\nabla\psi(x^k),x^{k+1}-x^k
\right\rangle
\leq
\frac{\beta}{2}\|{x^{k+1}-x^k}\|^2 .
\]
Thus,
\[
\begin{aligned}
&\psi(x^{k+1})-\psi(x^k)
+
\left\langle
\nabla\psi(x^{k+1}),u^k-x^{k+1}
\right\rangle
-
\left\langle
\nabla\psi(x^k),u^k-x^k
\right\rangle\\
=\;&
\psi(x^{k+1})-\psi(x^k)
-
\left\langle
\nabla\psi(x^k),x^{k+1}-x^k
\right\rangle\\
&+
\left\langle
\nabla\psi(x^{k+1})-\nabla\psi(x^k),
u^k-x^{k+1}
\right\rangle\\
\leq\;&
\frac{\beta}{2}\norm{x^{k+1}-x^k}^2
+
\beta\|{x^{k+1}-x^k}\|\|{x^{k+1}-u^k}\|\\
\leq\;&
\beta\|{x^{k+1}-x^k}\|^2
+
\frac{\beta}{2}\|{x^{k+1}-u^k}\|^2 
\end{aligned}
\]
using the elementary inequality $ab \le (a^2 + b^2)/2$ on the final line. The difference in the Lyapunov function after an update thus satisfies
\[
\begin{aligned}
&\Theta_{\gamma,\eta}(x^{k+1},u^{k+1},w^{k+1})
-
\Theta_{\gamma,\eta}(x^k,u^k,w^k)\\
\leq\;&
-\frac{1-l_\phi}{2\gamma}
\|{x^{k+1}-x^k}\|^2
+
\beta\|{x^{k+1}-x^k}\|^2\\
&+
\left(
\frac1\gamma+\frac{\beta}{2}
\right)
\|{x^{k+1}-u^k}\|^2
-
\frac{1-\eta}{\gamma}\|{x^k-u^k}\|^2 \\
&= -\frac{1-l_\phi}{2\gamma} \|\delta\|^2 + \beta \|\delta\|^2 + \left(\frac{1}{\gamma} + \frac{\beta}{2}\right) \|x^{k+1} - u^k\|^2 - \frac{1-\eta}{\gamma} \|x^k - u^k\|^2
\end{aligned}.
\]
Define the constant $A \coloneqq 1/\gamma + \beta/2$, and define the latter two terms as $\Xi$,
\begin{align*}
    \Xi &\coloneqq A \|x^{k+1} - u^k\|^2 - \frac{1-\eta}{\gamma} \|x^k - u^k\|^2\\
    &= A \|\delta - \frac{1}{\eta}s\|^2 - \frac{1-\eta}{\gamma \eta^2} \|s\|^2\\
    &= A\|\delta\|^2 - \frac{2A}{\eta}\langle \delta, s\rangle + \left(\frac{A}{\eta^2} - \frac{1-\eta}{\gamma \eta^2}\right) \|s\|^2\\
    &= A\|\delta\|^2 - \frac{2A}{\eta}\langle \delta, s\rangle + \frac{1}{\eta^2}\left(\frac{\beta}{2} + \frac{\eta}{\gamma}\right) \|s\|^2.
\end{align*}
Rearranging into components with and without $\beta$ yields
\begin{align*}
    \Xi &= \left[\frac{1}{\gamma} \|\delta\|^2 - \frac{2}{\gamma \eta} \langle \delta, s\rangle + \frac{1}{\gamma \eta} \|s\|^2\right] + \left[\frac{\beta}{2} \|\delta\|^2 - \frac{\beta}{\eta}\langle \delta, s \rangle + \frac{\beta}{2\eta^2}\|s\|^2 \right]\\
    &= \frac{1}{\gamma} \|\delta\|^2 - \frac{1}{\gamma \eta}\|\delta\|^2 + \frac{1}{\gamma \eta} \|s-\delta\|^2 + \frac{\beta}{2\eta^2} \|s-\eta \delta\|^2 \\
    &= \left(\frac{1}{\gamma} - \frac{1}{\gamma \eta}\right)\|\delta\|^2 + \frac{1}{\gamma \eta}\|d\|^2 + \frac{\beta}{2\eta^2} \|(1-\eta)\delta + d\|^2 \\
    &\le \left(\frac{1}{\gamma} - \frac{1}{\gamma \eta}\right) \|\delta\|^2 + \frac{L_\phi^2}{\eta \gamma } \|\delta\|^2 + \frac{\beta (1-\eta+ L_\phi)^2}{2\eta^2} \|\delta\|^2,
\end{align*}
where the latter inequalities come from $s = \delta + d$, triangle inequality, $\nabla \phi_1$ being $L_\phi$-Lipschitz, and $\eta \le 1$. Combining yields the Lyapunov decay
\begin{align*}
    &\quad \Theta_{\gamma,\eta}(x^{k+1},u^{k+1},w^{k+1})-
\Theta_{\gamma,\eta}(x^k,u^k,w^k)\\
& \le \left(-\frac{1-l_\phi}{2\gamma} + \beta + \frac{1}{\gamma} - \frac{1}{\gamma \eta} + \frac{L_\phi^2}{\eta \gamma} + \frac{\beta (1 - \eta + L_\phi)^2}{2\eta^2}\right)\|\delta\|^2\\
&= -\Lambda(\gamma, \eta) \|\delta\|^2
\end{align*}
as desired.
\end{proof}

\color{black}

% Put the proof of Proposition~\ref{prop:descent} here.

\subsection{Additional supporting lemma}
In addition to a descent condition, another technical condition for convergence is a relative error bound. This simply states that if the difference between successive iterates is small, then the subgradient must also be small. The following lemma details this, and will be used in the proof of \Cref{thm:global_convergence}.
\begin{lemma}\label{lem:relative-error}
Assume that the assumptions of Proposition \ref{prop:descent} hold. Then there
exists a constant $C>0$ such that, for all $k\geq 1$,
\[
    \dist
    \left(
        0,\partial\Theta_{\gamma,\eta}(x^k,u^k,w^k)
    \right)
    \leq
    C\norm{x^{k+1}-x^k}.
\]
\end{lemma}
\begin{proof}
The proof mainly follows the argument of Lemma B.1 in \cite{bian2021three}, where we only use the assumption that $\nabla\psi$ is
$\beta$-Lipschitz continuous. Notice that we have the
equivalent expression for $\Theta_{\gamma,\eta}$:
\[
\begin{aligned}
\Theta_{\gamma,\eta}(x,u,w)
=&\;
\frac{1}{\gamma}\pso(x)
+
\frac{1}{\gamma}\pst(u)
+
\psi(x)\\
&+
\frac{1-2\eta}{2\gamma}\norm{x-u}^2
-
\frac{1}{\gamma}\iprod{x-u}{w-x}
+
\iprod{\nabla\psi(x)}{u-x}.
\end{aligned}
\]
Let $s^k:=w^k-w^{k-1}$, we now estimate the partial subgradients of
$\Theta_{\gamma,\eta}$. First, with respect to $w$, we have
\[
    \nabla_w\Theta_{\gamma,\eta}(x^k,u^k,w^k)
    =
    \frac{1}{\gamma}(u^k-x^k)
    =
    \frac{1}{\gamma\eta}s^k .
\]
Hence
\[
    \norm{\nabla_w\Theta_{\gamma,\eta}(x^k,u^k,w^k)}
    \leq
    \frac{1}{\gamma\eta}\norm{s^k}.
\]

Next, let $\xi^k\in\partial\pst(u^k)$. From the definition of proximal operator, there holds
$\xi^k+u^k-2x^k+w^{k-1}+\gamma\nabla\psi(x^k)=0$. Then, with \eqref{eq:sk}, we have
\[
\begin{aligned}
\partial_u\Theta_{\gamma,\eta}(x^k,u^k,w^k)
&\ni
\frac{1}{\gamma}\xi^k
+
\frac{1-2\eta}{\gamma}(u^k-x^k)
+
\frac{1}{\gamma}(w^k-x^k)
+
\nabla\psi(x^k)\\
&=
\frac{1}{\gamma}
\left[
w^k-w^{k-1}
+
2\eta(x^k-u^k)
\right]
=
-\frac{1}{\gamma}s^k .
\end{aligned}
\]
Therefore,
\[
\operatorname{dist}
\left(
0,
\partial_u\Theta_{\gamma,\eta}(x^k,u^k,w^k)
\right)
\leq
\frac{1}{\gamma}\norm{s^k}.
\]

Finally, define
\[
q_k(x):=\left\langle \nabla\psi(x),u^k-x\right\rangle.
\]
Then there exists
\[
\zeta^k
\in
\partial_x q_k(x^k)
\]
such that
\[
\zeta^k+\nabla\psi(x^k)
=
B^k(u^k-x^k),
\]
where $B^k$ is an element of the generalized Jacobian of $\nabla\psi$ at $x^k$.
The $\beta$-Lipschitz continuity of $\nabla\psi$ implies
\[
\norm{B^k}\leq \beta.
\]
Therefore,
\[
\norm{\zeta^k+\nabla\psi(x^k)}
\leq
\beta\norm{u^k-x^k}.
\]
Using \eqref{eq:sk} and $\|s^k\| = \|\delta + d\| \leq (1+L_\phi)\|x^{k+1}-x^k\|$, we obtain
\[
\begin{aligned}
\partial_x\Theta_{\gamma,\eta}(x^k,u^k,w^k)
&\ni
\frac{1}{\gamma}
\left(
\nabla\pso(x^k)+x^k-w^k
\right)
+
\frac{2(1-\eta)}{\gamma}(x^k-u^k)
+
\nabla\psi(x^k)+\zeta^k\\
&=
-\frac{1}{\gamma}s^k
-
\frac{2(1-\eta)}{\gamma\eta}s^k
+
\nabla\psi(x^k)+\zeta^k .
\end{aligned}
\]
Therefore,
\[
\operatorname{dist}
\left(
0,\partial_x\Theta_{\gamma,\eta}(x^k,u^k,w^k)
\right)
\leq
\left(
\frac{1}{\gamma}
+
\frac{2(1-\eta)}{\gamma\eta}
+
\frac{\beta}{\eta}
\right)
\norm{s^k}.
\]

Combining the estimates for all the components, there exists a
constant $C_0>0$ such that
\[
    \dist
    \left(
        0,\partial\Theta_{\gamma,\eta}(x^k,u^k,w^k)
    \right)
    \leq
    C_0\norm{s^k} \le C_0(1+L_\phi)\norm{x^{k+1}-x^k} 
\]
This completes the proof.
\end{proof}

\subsection{Proof of Theorem~\ref{thm:global_convergence}}

\begin{proof}
By Proposition \ref{prop:descent}, the PnP-DYS energy function
$\Theta_{\gamma,\eta}(x^k,u^k,w^k)$ satisfies the sufficient descent inequality
\[
\Theta_{\gamma,\eta}(x^{k+1},u^{k+1},w^{k+1})
-
\Theta_{\gamma,\eta}(x^k,u^k,w^k)
\leq
-\Lambda(\gamma,\eta)\norm{x^{k+1}-x^k}^2 .
\]
Since $\Lambda(\gamma,\eta)>0$, the sequence
$\{\Theta_{\gamma,\eta}(x^k,u^k,w^k)\}_{k\geq0}$ is nonincreasing.
Together with the lower semicontinuity of $\Theta_{\gamma,\eta}$ along the
generated sequence and following the proof of \cite[Theorem 3.7]{bian2021three},
we obtain that, for any cluster point $(\xsol,\usol,\wsol)$,
\[
    \Theta^\ast
    :=
    \lim_{k\to\infty}
    \Theta_{\gamma,\eta}(x^k,u^k,w^k)
    =
    \Theta_{\gamma,\eta}(\xsol,\usol,\wsol).
\]

Define $E_k:=\Theta_{\gamma,\eta}(x^k,u^k,w^k)$. If $E_k=\Theta^\ast$ is a fixed point for some $k$, the finite-length property holds trivially. Therefore, we only consider the case where $E_k>\Theta^\ast$ for all sufficiently large $k$.

By Lemma \ref{lem:relative-error}, there exists $C>0$ such that
\[
    \dist
    \left(
        0,\partial\Theta_{\gamma,\eta}(x^k,u^k,w^k)
    \right)
    \leq
    C\norm{x^{k+1}-x^k}.
\]
Since $\Theta_{\gamma,\eta}$ is a KL function, following the proof in
\cite[Theorem 3.7]{bian2021three}, for all sufficiently large $k$, the KL
inequality holds along the tail of the sequence. 
Using the concavity of the desingularizing function $\Psi$, we obtain
\[
\begin{aligned}
&\Psi(E_k-\Theta^\ast)-\Psi(E_{k+1}-\Theta^\ast)\\
&\geq
\Psi'(E_k-\Theta^\ast)(E_k-E_{k+1})\\
&\geq
\frac{E_k-E_{k+1}}
{
\dist
\left(
0,\partial\Theta_{\gamma,\eta}(x^k,u^k,w^k)
\right)
}\\
&\geq
\frac{\Lambda(\gamma,\eta)\norm{x^{k+1}-x^k}^2}{C \norm{x^{k+1}-x^k}}
=
\frac{\Lambda(\gamma,\eta)}{C}\norm{x^{k+1}-x^k} ,
\end{aligned}
\]
%where the case $\norm{x^{k+1}-x^k}=0$ is trivial. 
Therefore,
\[
    \norm{x^{k+1}-x^k}
    \leq
    \frac{C}{\Lambda(\gamma,\eta)}
    \left[
    \Psi(E_k-\Theta^\ast)-\Psi(E_{k+1}-\Theta^\ast)
    \right].
\]
Summing over $k$ yields
\[
    \sum_{k=0}^{\infty}\norm{x^{k+1}-x^k}<+\infty .
\]
The finite length property of $\{(w^k)\}_{k\geq 0}$ and $\{(u^k)\}_{k\geq 0}$ can be obtained from \eqref{eq:sk} and $\|s^k\| \leq (1+L_\phi)\|{x^{k+1}-x^k}\|$. 

It remains to show criticality of the cluster points. Suppose that we have convergence $(x^k, u^k, w^k) \rightarrow (x^*, u^*, w^*)$. The update for $w^k$ in \labelcref{maineq:pnp-dys-prox} immediately gives that $u^* = x^*$. Moreover, the optimality condition $x^{k+1} = \prox_{\phi_1}$ gives $w^k - x^{k+1} = \nabla \phi_1(x^{k+1})$. Passing to the limit yields $w^* - x^* = \nabla \phi_1(x^*)$.

Now define $q^k = 2x^{k+1} - w^k - \gamma \nabla \phi(x^{k+1})$ such that $u^{k+1} \in \prox_{\phi_2} (q^k)$. Then $q^k \rightarrow q^* = 2x^* - w^* - \gamma \nabla \psi(x^*)$. Since $\phi_2$ is proper and closed, its proximal also has a closed graph, and therefore $x^* = u^* \in \prox_{\phi_2}(q^*)$. The optimality condition therefore gives
\begin{equation*}
    0 \in \partial \phi_2(x^*) + x^* - q^*.
\end{equation*}
Substituting the definition of $q^*$ and using $w^* - x^* = \nabla \phi_1(x^*)$ yields
\begin{equation*}
    0 \in \partial \phi_2(x^*) + \nabla \phi_1(x^*) + \gamma \nabla \psi(x^*)
\end{equation*}
as desired.
\end{proof}

\section{Experimental details}
\subsection{GS-PnP and DPIR}\label{app:PnPMethods}
A critical distinction among the evaluated methods is the treatment of the data fidelity term. We divide the methods into two categories based on the data fidelity used:

\begin{enumerate}
\item Single-prior optimization with mixed-noise fidelity: 

We consider the PnP-PGD method, which incorporates the mixed-noise fidelity but focuses on the single-prior objective:
\beq\label{eq:obj-pgd}
\min_x\ h_2(x):=\psi(x)+\sfrac{1}{\gamma}\phi_{\sigma}(x)
\eeq
The iteration is formulated as:
\[
\left\{
\begin{aligned}
&\zk = \xk - \gamma \nabla \psi(\xk) \\
&\xkp =\gD_{\sigma}(\zk)
\end{aligned}
\right.
\]
where $\gD_{\sigma}=\prox_{\phi_{\sigma}}$.

\item Single-prior optimization with approximate Gaussian fidelity:

For PnP-DRS, GS-PnP and DPIR, integrating the mixed-noise fidelity $\psi(x)$ is computationally difficult because they require evaluating $\prox_{\psi}$. Consequently, we have to use the classical Gaussian data fidelity for these methods. Define
$\psi_g(x)=\frac{\lamg}{2}\norm{y-Ax}^2$, 
We consider the simplified objective:
\beq\label{eq:obj-gs}
\min_x\ h_3(x):=\psi_g(x) + \sfrac{1}{\gamma}\phi_{\sigma}(x)
\eeq
The iterative schemes for solving \eqref{eq:obj-gs} are listed below:

GS-PnP:
\[
\begin{aligned}
    &\text{find } \tau_k \in \{\tau, \eta\tau, \eta^2\tau, \dots\} \text{ satisfying } h_3(\xk) - h_3(\xkp) \ge \frac{c}{\tau_k} \norm{\xk - \xkp}^2: \\
    &\left\{
    \begin{aligned}
        \zk &= (\tau_k/\gamma) \gD_{\sigma}(\xk) + (1- \tau_k/\gamma) \xk \\
        \xkp &=\prox_{\tau_k \psi_g}(\zk)
\end{aligned}\right.\quad ,\eta\in(0,1),c\in(0,1/2)
\end{aligned}
\]
PnP-DRS:
\[
\left\{\begin{aligned}
&\ykp = \gD_{\sigma}(\xk)\\
&\zkp =\prox_{\gamma\psi_g}(2\ykp-\xk) \\
&\xkp = \xk + \zkp-\ykp
\end{aligned}\right. 
\]
DPIR: %\textcolor{red}{Fix this part with the logspacing}
\[
\left\{\begin{aligned}
&\ykp = \prox_{\frac{\gamma\sigma^2}{\sigma_k^2}\psi_g}(\xk)\\
&\xkp = \gD_{\sigma_k}(\ykp)\end{aligned}\right.
\]
where $\{\sigma_k\}_{k=1}^K$ is a sequence of noise levels decreasing monotonically via logarithmic spacing.
\end{enumerate}

\subsection{Camera kernels for deblurring}
The blur kernels are given by \Cref{fig:all_kernels}. Kernels 1 to 8 are real-world camera shake kernels, kernel 9 is a $9 \times 9$ uniform kernel and kernel 10 is a $25 \times 25$ Gaussian kernel with standard deviation $1.6$.

\begin{figure}[!hbtp]
    \centering
    \includegraphics[width=0.13\textwidth]{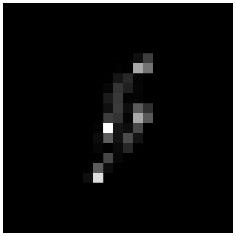} 
    \includegraphics[width=0.13\textwidth]{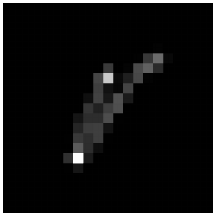} 
    \includegraphics[width=0.13\textwidth]{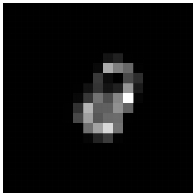} 
    \includegraphics[width=0.13\textwidth]{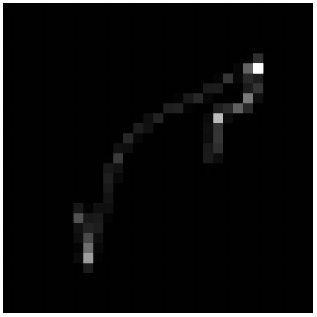} 
    \includegraphics[width=0.13\textwidth]{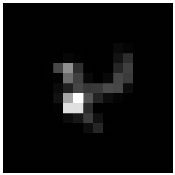} 
    \\
    \includegraphics[width=0.13\textwidth]{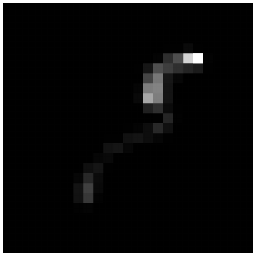} 
    \includegraphics[width=0.13\textwidth]{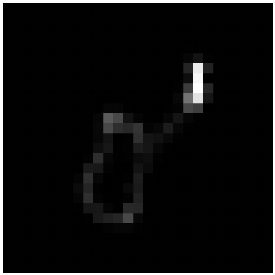} 
    \includegraphics[width=0.13\textwidth]{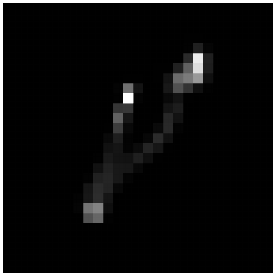} 
    \includegraphics[width=0.13\textwidth]{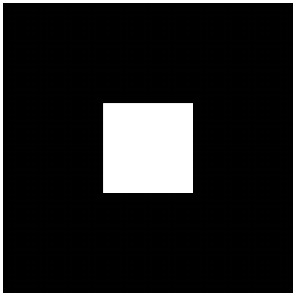} 
    \includegraphics[width=0.13\textwidth]{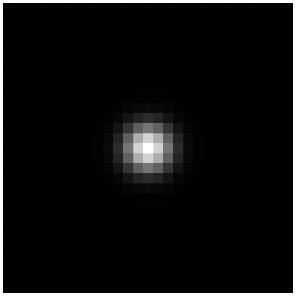}
    
    \caption{The 10 blur kernels used in our experiments. } 
    \label{fig:all_kernels}
\end{figure}

\section{Hyperparameter selection}\label{sec:params}

% \subsection{Hyperparameter selection}\label{sec:hyper-selection}
To determine the optimal hyperparameters for the image deblurring and super-resolution tasks, we first carry out grid-search experiments on the set3c dataset. For each parameter configuration, we evaluate the reconstruction performance across 8 distinct blur kernels for the image deblurring task, and across downsampling scale factors of $s=2,3$, then select the hyperparameters that yield the highest average PSNR.

To ensure a fair comparison, we carefully optimize the hyperparameters for all considered algorithms. For problem \eqref{eq:obj-pgd} and \eqref{eq:obj-gs} and corresponding algorithms, we fix the algorithm step size $\gamma=1.0$ to comply with the standard convergence analysis of PnP methods.
Specifically, assuming the data fidelity term $\psi_g$ has an $L_{\psi_g}$-Lipschitz gradient with $L_{\psi_g}<1$, the convergence of PnP-PGD and PnP-DRS is theoretically guaranteed when $\lamg L_{\psi_g}<1$ \cite{hurault2023relaxed,hurault2024convergent}. The parameters for GS-PnP and DPIR are further configured following the settings in \cite{hurault2021gradient} and \cite{zhang2021plug}, respectively.

For PnP-DYS, the parameter selection is theoretically guided by \Cref{prop:descent}. We fix $\gamma=2$ so that the objective function of \eqref{eq:proxyF} is the same as \eqref{eq:obj-pgd} when $\phi_{\sgo}=\phi_{\sgt}=\phi_{\sigma}$. The relaxation parameter $\eta$ should be chosen to satisfy $\Lambda(\gamma,\eta)>0$ as defined in \eqref{eq:Gamma}; in our implementations, we fix $\eta=0.2$ to ensure stable algorithmic descent and reduce hyperparameter tuning requirements. We note that larger relaxation $\eta$ can likely be used if a relaxed denoiser is used to reduce the Lipschitz constant of $g$.

 To account for the mixed noise, we define the ``total noise level'', such as $\sigma_{\total} = \sigma_g + \sigma_l$ for the Laplace--Gaussian noise, where $\sigma_g$ and $\sigma_l$ correspond to the standard deviation of the Gaussian noise $\mathcal{N}(0, \sigma_g^2)$ and the scale parameter of the Laplace noise $\text{Lap}(0, \sigma_l)$, respectively. We then introduce the noise scale ratios $\rho_1 = \sigma_1 / \sigma_{\total}$ and $\rho_2 = \sigma_2 / \sigma_{\total}$, where $\sigma_1$ and $\sigma_2$ represent the denoising strengths applied to the corresponding denoisers.

For PnP-DYS, we optimize the fidelity weights (e.g., $\laml, \lamg$ or $\lamp, \lamg$) alongside two scale ratios $\rho_1, \rho_2$ for the dual denoisers. For PnP-PGD, the search involves the fidelity weights and a single scale ratio $\rho_1$. For GS-PnP and PnP-DRS, we tune the data fidelity weight $\lamg$ and the ratio $\rho_1$. For DPIR, we only need to select $\lamg$. The specific numerical search grids for each method vary across different noise types and the image deblurring and super-resolution tasks.

All hyperparameters are selected by grid search on the set3c dataset and then evaluated on CBSD10 dataset. In the following, we show the hyperparameter selections for Laplace--Gaussian mixed noise and Poisson--Gaussian mixed noise respectively.

\subsection{Laplace--Gaussian mixed noise}\label{sec:params-lg}

For the image deblurring task, the fidelity weights of PnP-DYS are searched within $\laml, \lamg \in \{1, 2, 3\}$, while the denoiser scale ratios are tuned over $\rho_1 \in \{0.6, 0.9, 1.2, 1.5, 1.8\}$ and $\rho_2 \in \{0.3, 0.6, 0.9\}$. For PnP-PGD, we explore the same fidelity weight grid alongside the single noise ratio $\rho_1$. GS-PnP is tuned with a weight grid of $\laml \in \{1, 2, 3, 4, 5, 6\}$. Finally, as PnP-DRS and DPIR exhibit significant convergence instability for the deblurring task even with small values of $\lamg$, so we empirically set $\lamg=0.1$ for both methods and $\rho_1=0.6$ for PnP-DRS to achieve the most stable results possible. 

For the super-resolution task, the grid search of $\rho_1$ is set to $\rho_1 \in \{0.6, 0.9, 1.2, 1.5, 1.8, 2.0\}$. We further search $\lamg\in\{0.1, 0.2, 0.4, 0.6, 0.8, 1.0\}$ for PnP-DRS and DPIR. For DRS, we set $\lamg=0.1$ and $\rho_1=0.6$. For DPIR, we set $\lamg=0.4$.
The results are summarized in Table \ref{tab:sisr_laplace}.
The hyperparameter selection results of PnP-DYS, PnP-PGD and GS-PnP are shown in Table~\ref{tab:dys-laplace}, Table~\ref{tab:pgd-laplace} and  Table~\ref{tab:gs-laplace} respectively.

For high noise deblurring, the hyperparameter grid search range is extended. For data fidelity parameters, we search within $\laml \in \{0.1, 0.5, 1.0, 2.0, 3.0\}$ and $\lamg \in \{0.5, 1.0, 2.0, 3.0\}$. The denoiser scale ratios $\rho_i$ follow the search strategy established in the previous sections. Under this criterion, the optimal configuration for PnP-DYS is identified as $\laml = 0.5$, $\lamg = 2.0$ with $\rho_1 = 1.2$ and $\rho_2 = 0.3$. Similarly, the optimal parameters for PnP-PGD are set to $\lambda_l = 0.5$, $\lambda_g = 1.0$ with $\rho = 0.6$, while for GS-PnP, the configuration is $\lambda = 1.0$ and $\rho = 0.6$.

\begin{table}[!htbp]
    \centering
    \caption{Hyperparameter selections of PnP-DYS for image deblurring and single-image super-resolution  tasks under Laplace--Gaussian mixed noise.}
    \label{tab:dys-laplace}
    \label{tab:dys-laplace}
    \begin{tabular}{lcccccccccc}
        \toprule
        & \multicolumn{4}{c}{Deblur} & \multicolumn{4}{c}{SR} \\
        \cmidrule(lr){2-5} \cmidrule(lr){6-9} 
        $\sigma_g$ & 2.55 & 2.55 & 7.65 & 7.65 & 2.55 & 2.55 & 7.65 & 7.65 \\
        $\sigma_l$ & 2.55 & 7.65 & 2.55 & 7.65 & 2.55 & 7.65 & 2.55 & 7.65 \\
        \midrule
        $\laml$ &  1 & 1 & 1 & 1 & 1& 1 & 1 & 1\\
        $\lamg$ & 3 & 3 & 3 & 3 & 3 & 2 & 3 & 3\\
        $\rho_1$ & 1.5 & 1.8 & 1.2 & 1.5 & 2.0 & 1.8 & 1.8 & 1.8 \\
        $\rho_2$ & 0.3 & 0.3 & 0.3 & 0.3 & 0.3 & 0.3 & 0.3 & 0.3\\
        \bottomrule
    \end{tabular}
\end{table}

\begin{table}[!htbp]
    \centering
    \caption{Hyperparameter selections of PnP-PGD for image deblurring and single-image super-resolution  tasks under Laplace--Gaussian mixed noise.}
    \label{tab:pgd-laplace}
    \begin{tabular}{lcccccccccc}
        \toprule
        & \multicolumn{4}{c}{Deblur} & \multicolumn{4}{c}{SR} \\
        \cmidrule(lr){2-5} \cmidrule(lr){6-9} 
        $\sigma_g$ & 2.55 & 2.55 & 7.65 & 7.65 & 2.55 & 2.55 & 7.65 & 7.65 \\
        $\sigma_l$ & 2.55 & 7.65 & 2.55 & 7.65 & 2.55 & 7.65 & 2.55 & 7.65 \\
        \midrule
        $\laml$ &  1 & 1 & 1 & 1 & 1& 1& 1& 1\\
        $\lamg$ & 2 & 2 & 2 & 2 & 3& 3& 3& 3\\
        $\rho_1$ & 0.9 & 0.9 & 0.9 & 0.9 & 1.5 & 1.5 & 1.2 & 1.2\\
        \bottomrule
    \end{tabular}
\end{table}

\begin{table}[!htbp]
    \centering
    \caption{Hyperparameter selections of GS-PnP for image deblurring and single-image super-resolution  tasks under Laplace--Gaussian mixed noise.}
    \label{tab:gs-laplace}
    \begin{tabular}{lcccccccccc}
        \toprule
        & \multicolumn{4}{c}{Deblur} & \multicolumn{4}{c}{SR} \\
        \cmidrule(lr){2-5} \cmidrule(lr){6-9} 
        $\sigma_g$ & 2.55 & 2.55 & 7.65 & 7.65 & 2.55 & 2.55 & 7.65 & 7.65 \\
        $\sigma_l$ & 2.55 & 7.65 & 2.55 & 7.65 & 2.55 & 7.65 & 2.55 & 7.65 \\
        \midrule
        $\lamg$ &  5 & 2 & 3 & 2 & 4 & 2 & 2 & 1\\
        $\rho_1$ & 1.2 & 0.9 & 0.9 & 0.6 & 1.5 & 0.9 & 0.9 & 0.6\\
        \bottomrule
    \end{tabular}
\end{table}

\subsection{Poisson--Gaussian mixed noise}
\label{sec:params-pg}
Recall from \cref{eqs:PoissonGaussianGrad} that the infimal convolution fidelity is given by 
\begin{equation*}
    \psi(x) = \min_v \left[\lambda_p\KL(v \|Ax) + \frac{\lambda_g}{2}\|y-v\|^2\right].
\end{equation*}
The fidelity weights are searched over
\[
    \lamp,\lamg \in \{0.5,1.0,2.0,3.0,5.0,10.0\}.
\]
For image deblurring, the selected PnP-DYS parameters are
$(\lamp,\lamg,\rho_1,\rho_2)=(5.0,5.0,1.2,0.3)$ for
$\sigma_p\in\{20,200\}$ and
$(\lamp,\lamg,\rho_1,\rho_2)=(2.0,5.0,0.9,0.3)$ for
$\sigma_p\in\{40,60,100\}$.
For PnP-PGD, we use
$(\lamp,\lamg,\rho)=(5.0,2.0,0.6)$ for all $\sigma_p$ levels.
For GS-PnP, we use $(\lambda,\rho)=(1.0,0.6)$ for
$\sigma_p\in\{20,40\}$ and $(\lambda,\rho)=(2.0,0.6)$ for
$\sigma_p\in\{60,100,200\}$.

For image super-resolution, the selected PnP-DYS parameters are
$(\lamp,\lamg,\rho_1,\rho_2)=(2.0,3.0,0.9,0.3)$ for
$\sigma_p\in\{20,40,60\}$ and
$(\lamp,\lamg,\rho_1,\rho_2)=(3.0,5.0,1.2,0.3)$ for
$\sigma_p\in\{100,200\}$.
For PnP-PGD, we use
$(\lamp,\lamg,\rho)=(2.0,3.0,0.6)$ for $\sigma_p=20$ and
$(\lamp,\lamg,\rho)=(3.0,5.0,0.9)$ for
$\sigma_p\in\{40,60,100,200\}$.
For GS-PnP, we use $(\lambda,\rho)=(1.0,0.6)$ for $\sigma_p=20$ and
$(\lambda,\rho)=(2.0,0.6)$ for
$\sigma_p\in\{40,60,100,200\}$.

\section{Ablation study on hyperparameters}\label{sec:ablation}
% \begin{enumerate}
%     \item Check sensitivity of methods to the choice of $\lambda, \sigma_d$
%     \item Can be done for one fixed noise level, say $\sigma_g, \sigma_l = 0.03$, 
%     \item Same for high noise scenario for TV denoiser. 
% \end{enumerate}

In this section, we study the hyperparameter sensitivity of the proposed method for the image deblurring problem under Laplace--Gaussian mixed noise. 

In our PnP-DYS framework, we have four key parameters: $\laml, \lamg$ for the data fidelity terms and $\rho_1, \rho_2$ for the noise level of denoisers. It is interesting to investigate how their combinations influence the reconstruction effectiveness. To this end, we conduct a sensitivity analysis under the following scenario: the blurred image is corrupted by mixed noise consisting of $3\%$ Gaussian and $3\%$ Laplace components.
\begin{enumerate}[label={\rm (\roman*)}, ref={\rm (\roman*)}]
    \item $\laml$ and $\lamg$: We first fix $\rho_1=1.2$ and $\rho_2=0.3$ to investigate the sensitivity of the performance with respect to $\laml,\lamg\in[0.1, 0.5, 1.0, 1.5, 2.0, 2.5, 3.0]$. As shown in Table~\ref{tab:ablation_lambda}, $\laml$ and $\lamg$ control the balance between the Gaussian and Laplacian noise priors. We observe that the best performance is achieved when $(\laml, \lamg) = (0.5, 2.5)$ with a peak PSNR of 27.64 dB.  Notably, the performance remains remarkably stable for $\laml \in [0.5, 3.0]$ and $\lamg \in [2.0, 3.0]$, suggesting that PnP-DYS framework is quite robust for a wide range of settings.
    %$\laml/\lamg$ matches the ratio of the noise levels \item $\sigma_g/\sigma_l$. A significant deviation from this ratio leads to either oversmoothing or residual artifacts.
    \item $\rho_1$ and $\rho_2$: We then fix $\laml=1$ and $\lamg=3$ and perform a grid search on $\rho_1,\rho_2\in [0.3, 0.6, 0.9, 1.2, 1.5, 1.8]$. The experimental results in Table~\ref{tab:ablation_rho} show that the peak performance is achieved when one of the parameters is set to a relatively small value while the other is moderately larger. Specifically, a maximum PSNR of 27.62 dB is obtained at $(\rho_1, \rho_2) = (0.3, 1.5)$. Interestingly, the table displays a symmetric-like performance: for instance, the combination $(1.5, 0.3)$ also yields a high PSNR of 27.60 dB. This result is expected, because under this symmetric parameter settings, the underlying optimization problem remains invariant, thereby the resulting solutions should be the same. In practice, we always set $\rho_1$ to be relatively large while $\rho_2$ is kept smaller. 
\end{enumerate}

\begin{table}[!htbp]
    \centering
    \caption{PSNR (dB) results for different data fidelity parameters $\laml$ and $\lamg$ for Laplace--Gaussian noise CBSD10 dataset under kernel $2$, $\sigma_g=\sigma_l=7.65$ and fixed $\rho_1=1.5, \rho_2=0.3$. Observe that the reconstruction does not change as $\lambda_l$ changes for smaller values of $\lambda_g$. This is explained by \cref{eq:LaplaceGaussianFidelityGrad}: $\lambda_l$ enters only through the proximal step. Equal reconstructions indicate that the optimal slack variable $\bar{w}$ is very close to $y$ and the proximal is likely clamping all pixels.}
    \label{tab:ablation_lambda}
    \begin{tabular}{c|ccccccc}
        \toprule
        $\laml \setminus \lamg$ & 0.1 & 1.0 & 2.0 & 3.0 & 4.0 & 5.0 & 10.0\\
        \midrule
        0.1 & 23.23 & 25.93 & 26.52 & 26.66 & 26.72 & 26.77 & 26.78\\
        1.0 & 23.23 & 26.03 & 27.12 & 27.60 & 27.54 & 25.70 & 1.60\\
        2.0 & 23.23 & 26.03 & 27.12 & 27.60 & 27.53 & 25.52 & -6.00\\
        3.0 & 23.23 & 26.03 & 27.12 & 27.60 & 27.53 & 25.52 & -10.42\\
        4.0 & 23.23 & 26.03 & 27.12 & 27.60 & 27.53 & 25.52 & -13.52\\
        5.0 & 23.23 & 26.03 & 27.12 & 27.60 & 27.53 & 25.52 & -15.81\\
        10.0 & 23.23 & 26.03 & 27.12 & 27.60 & 27.53 & 25.52 &-22.40\\
        \bottomrule
    \end{tabular}
\end{table}

\begin{table}[!htbp]
    \centering
    \caption{PSNR (dB) results for different denoising levels $\rho_1$ and $\rho_2$ on CBSD10 dataset under kernel $2$, $\sigma_g=\sigma_l=7.65$ and fixed $\laml=1, \lamg=3$. While the target functionals are identical when flipping $\rho_1,\rho_2$, the DYS splitting is not symmetric, and an additional asymmetry arises from the final application of the first denoiser $\gD_1$ for the final reconstruction.}
    \label{tab:ablation_rho}
    \begin{tabular}{c|cccccc}
        \toprule
        $\rho_1 \setminus \rho_2$ & 0.3 & 0.6 & 0.9 & 1.2 & 1.5 & 1.8\\
        \midrule
        0.3 & 10.33 & 9.86 & 23.08 & 27.45 & 27.62 & 27.37 \\
        0.6 & 9.88 & 19.13 & 27.01 & 27.59 & 27.45 & 27.17\\
        0.9 & 23.11 & 27.05 & 27.51 & 27.31 & 27.15 & 26.93\\
        1.2 & 27.45 & 27.61 & 27.39 & 26.93 & 26.57 & 26.61\\
        1.5 & 27.60 & 27.46 & 27.19 & 26.81 & 26.18 & 25.91\\
        1.8 & 27.37 & 27.18 & 26.94 & 26.69 & 26.16 & 25.64\\
        \bottomrule
    \end{tabular}
\end{table}

\section{Experimental Results}\label{sec:ex-results}
For completeness, here we report the detailed PSNR values over all 10 blur kernels and all considered Laplace--Gaussian noise levels for the deblurring task.

\begingroup
\scriptsize
\setlength{\tabcolsep}{4pt}
\renewcommand{\arraystretch}{1.08}

\begin{longtable}{cccccc}
    \caption{Performance comparison table of deblurring for Laplace--Gaussian mixed noise. Average PSNR over CBSD10 images across different deblurring kernels and noise levels. The best results are highlighted in bold, and the second-best results are underlined.}
    \label{tab:performance_final_pro}\\
    \toprule
    \multirow{3}{*}{Methods} & \multirow{3}{*}{$k$} & \multicolumn{4}{c}{($\sigma_g, \sigma_l$)}  \\
    \cmidrule(lr){3-6}
    & & (2.55, 2.55) & (2.55, 7.65) & (7.65, 2.55) & (7.65, 7.65)  \\
    \midrule
    \endfirsthead

    \caption[]{Performance comparison table of deblurring for Laplace--Gaussian mixed noise (continued).}\\
    \toprule
    \multirow{3}{*}{Methods} & \multirow{3}{*}{$k$} & \multicolumn{4}{c}{($\sigma_g, \sigma_l$)} \\
    \cmidrule(lr){3-6}
    & & (2.55, 2.55) & (2.55, 7.65) & (7.65, 2.55) & (7.65, 7.65)  \\
    \midrule
    \endhead

    \endfoot

    \bottomrule
    \endlastfoot

    \dysgs & \multirow{6}{*}{1} & \underline{31.65} & \underline{28.41} & \underline{29.15} & \underline{27.84}  \\
    \dystv &  & \textbf{31.69} & \textbf{28.51} & \textbf{29.37} & \textbf{27.91}  \\
    PnP-PGD & & 30.76 & 27.82 & 28.57 & 27.36  \\
    GS-PnP  & & 30.74 & 27.48 & 28.25 & 26.57  \\
    PnP-DRS  & & 24.75 & 22.39 & 23.35 & 21.63  \\
    DPIR  & & 24.19 & 23.81 & 24.07 & 23.95  \\
    \midrule

    \dysgs & \multirow{6}{*}{2} & \textbf{31.10} & \textbf{28.17} & \underline{28.98} & \textbf{27.60}  \\
    \dystv &  & \underline{31.05} & \underline{28.16} & \textbf{29.04} & \underline{27.55}  \\
    PnP-PGD & & 30.41 & 27.87 & 28.33 & 27.24  \\
    GS-PnP  & & 30.56 & 27.48 & 28.20 & 26.81  \\
    PnP-DRS  & & 20.55 & 14.90 & 16.59 & 13.80  \\
    DPIR  & & 15.50 & 9.14 & 11.12 & 7.80  \\
    \midrule

    \dysgs & \multirow{6}{*}{3} & \textbf{31.12} & \textbf{28.39} & \textbf{29.24} & \textbf{27.77}  \\
    \dystv &  & \underline{31.06} & \underline{28.32} & \underline{29.23} & \underline{27.73}  \\
    PnP-PGD & & 30.46 & 28.19 & 28.55 & 27.63  \\
    GS-PnP  & & 30.77 & 27.93 & 28.60 & 27.40  \\
    PnP-DRS  & & 20.02 & 14.64 & 16.23 & 13.59  \\
    DPIR  & & 16.08 & 8.53 & 10.70 & 7.12  \\
    \midrule

    \dysgs & \multirow{6}{*}{4} & \textbf{30.86} & \underline{27.84} & \underline{28.64} & \textbf{27.26}  \\
    \dystv &  & \underline{30.83} & \textbf{27.85} & \textbf{28.76} & \underline{27.25}  \\
    PnP-PGD & & 30.06 & 27.43 & 27.94 & 26.85  \\
    GS-PnP  & & 30.33 & 27.16 & 27.92 & 26.48  \\
    PnP-DRS  & & 20.37 & 14.69 & 16.38 & 13.57  \\
    DPIR  & & 15.61 & 9.09 & 11.12 & 7.75  \\
    \midrule

    \dysgs & \multirow{6}{*}{5} & \textbf{32.42} & \underline{29.20} & \textbf{29.93} & \underline{28.72}  \\
    \dystv & & \underline{32.35} & 29.08 & 29.77 & 28.54  \\
    PnP-PGD & & 31.92 & \textbf{29.29} & \underline{29.92} & \textbf{28.75}  \\
    GS-PnP  & & 31.68 & 28.79 & 29.46 & 27.93  \\
    PnP-DRS  & & 21.13 & 15.68 & 17.34 & 14.60  \\
    DPIR  & & 16.94 & 9.36 & 11.55 & 7.95  \\
    \midrule

    \dysgs & \multirow{6}{*}{6} & \textbf{32.44} & \textbf{29.20} & \textbf{29.97} & \textbf{28.71}  \\
    \dystv &  & \underline{32.42} & \underline{29.13} & 29.72 & \underline{28.60}  \\
    PnP-PGD & & 32.03 & 29.12 & \underline{29.79} & 28.51  \\
    GS-PnP  & & 31.60 & 28.56 & 29.28 & 27.62  \\
    PnP-DRS  & & 21.41 & 15.68 & 17.40 & 14.55  \\
    DPIR  & & 16.24 & 9.34 & 11.41 & 7.95  \\
    \midrule

    \dysgs & \multirow{6}{*}{7} & \textbf{31.14} & \textbf{28.36} & \textbf{29.11} & \textbf{27.95}  \\
    \dystv &  & 31.02 & 28.20 & 28.80 & 27.75  \\
    PnP-PGD & & \underline{30.66} & \underline{28.28} & \underline{28.86} & \underline{27.87}  \\
    GS-PnP  & & \underline{30.66} & 27.98 & 28.63 & 27.32  \\
    PnP-DRS  & & 20.28 & 14.99 & 16.57 & 13.95  \\
    DPIR  & & 15.40 & 8.78 & 10.81 & 7.44  \\
    \midrule

    \dysgs & \multirow{6}{*}{8} & \textbf{30.83} & \textbf{28.22} & \textbf{28.95} & \textbf{27.76}  \\
    \dystv &  & \underline{30.74} & \underline{28.17} & \underline{28.86} & \underline{27.68}  \\
    PnP-PGD & & 30.17 & 28.02 & 28.35 & 27.47  \\
    GS-PnP  & & 30.42 & 27.70 & 28.36 & 27.03  \\
    PnP-DRS  & & 19.99 & 14.60 & 16.19 & 13.54  \\
    DPIR  & & 14.99 & 8.41 & 10.43 & 7.06  \\
    \midrule

    \dysgs & \multirow{6}{*}{9} & \textbf{27.41} & \textbf{26.14} & \textbf{26.63} & \textbf{25.82}  \\
    \dystv &  & \underline{27.39} & \underline{26.11} & \underline{26.62} & 25.77  \\
    PnP-PGD & & 27.07 & 26.03 & 26.20 & 25.64  \\
    GS-PnP  & & 27.37 & 25.96 & 26.38 & \underline{25.80}  \\
    PnP-DRS  & & 19.02 & 14.31 & 15.77 & 13.31  \\
    DPIR  & & 17.01 & 11.18 & 13.06 & 9.91  \\
    \midrule

    \dysgs & \multirow{6}{*}{10} & \underline{28.91} & \textbf{27.71} & \underline{28.12} & \textbf{27.47}  \\
    \dystv & &  \textbf{28.97} & \underline{27.70} & \textbf{28.13} & \underline{27.38}   \\
    PnP-PGD & & 28.49 & 27.60 & 27.83 & 27.35  \\
    GS-PnP  & & 28.79 & 27.64 & 27.97 & 27.32  \\
    PnP-DRS  & & 26.21 & 23.69 & 24.70 & 22.88  \\
    DPIR  & & 26.35 & 25.59 & 26.20 & 25.96  \\
    \bottomrule
\end{longtable}

\endgroup




\end{appendices}


\bibliography{refs}% common bib file

\end{document}